\documentclass[11pt,a4paper]{article}
\usepackage{amsmath, amsthm, amsfonts, amssymb, color}
\usepackage{mathrsfs}
\usepackage{color}
\usepackage{bbm}

\newtheorem{thm}{Theorem}[section]
\newtheorem{cor}[thm]{Corollary}
\newtheorem{lem}[thm]{Lemma}

\newtheorem{rem}[thm]{Remark}
\theoremstyle{definition}

\newcommand{\scr}[1]{\mathscr #1}
\definecolor{wco}{rgb}{0.5,0.2,0.3}

\numberwithin{equation}{section} \theoremstyle{remark}

\newcommand{\ua}{\uparrow}

\title{{\bf   Path-Distribution Dependent SDEs: Well-Posedness and Asymptotic Log-Harnack Inequality}\footnote{Supported in part by   the National Key R\&D Program of China (No. 2022YFA1006000, 2020YFA0712900).} }
\author{
	{\bf   Feng-Yu Wang$^{a)}$, Chenggui Yuan$^{b)}$, Xiao-Yu Zhao$^{a)}$  }\\
	\footnotesize{ a)Center for Applied Mathematics, Tianjin
		University, Tianjin 300072, China}\\
	\footnotesize{wangfy@tju.edu.cn,\  zhxy\_0628@tju.edu.cn}\\
	\footnotesize{ b)Department of Mathematics, Swansea University, Bay campus, SA1 8EN, UK}\\
	\footnotesize{ c.yuan@swansea.ac.uk }
}
\begin{document}
	\allowdisplaybreaks
	\def\R{\mathbb R}  \def\ff{\frac} \def\ss{\sqrt} \def\B{\mathbf
		B}\def\TO{\mathbb T}
	\def\I{\mathbb I_{\pp M}}\def\p<{\preceq}
	\def\N{\mathbb N} \def\kk{\kappa} \def\m{{\bf m}}
	\def\ee{\varepsilon}\def\ddd{D^*}
	\def\dd{\delta} \def\DD{\Delta} \def\vv{\varepsilon} \def\rr{\rho}
	\def\<{\langle} \def\>{\rangle} \def\GG{\Gamma} \def\gg{\gamma}
	\def\nn{\nabla} \def\pp{\partial} \def\E{\mathbb E}
	\def\d{\text{\rm{d}}} \def\bb{\beta} \def\aa{\alpha} \def\D{\scr D}
	\def\si{\sigma} \def\ess{\text{\rm{ess}}}
	\def\beg{\begin} \def\beq{\begin{equation}}  \def\eed{\end{equation}}\def\F{\scr F}
	\def\Ric{{\rm Ric}} \def\Hess{\text{\rm{Hess}}}
	\def\e{\text{\rm{e}}} \def\ua{\underline a} \def\OO{\Omega}  \def\oo{\omega}
	\def\tt{\tilde}
	\def\cut{\text{\rm{cut}}} \def\P{\mathbb P} \def\ifn{I_n(f^{\bigotimes n})}
	\def\C{\scr C}      \def\aaa{\mathbf{r}}     \def\r{r}
	\def\gap{\text{\rm{gap}}} \def\prr{\pi_{{\bf m},\varrho}}  \def\r{\mathbf r}
	\def\Z{\mathbb Z} \def\vrr{\varrho} \def\ll{\lambda}
	\def\L{\scr L}\def\Tt{\tt} \def\TT{\tt}\def\II{\mathbb I}
	\def\i{{\rm in}}\def\Sect{{\rm Sect}}  \def\H{\mathbb H}
	\def\M{\scr M}\def\Q{\mathbb Q} \def\texto{\text{o}} \def\LL{\Lambda}
	\def\Rank{{\rm Rank}} \def\B{\scr B} \def\i{{\rm i}} \def\HR{\hat{\R}^d}
	\def\to{\rightarrow}\def\l{\ell}\def\iint{\int}
	\def\EE{\scr E}\def\Cut{{\rm Cut}}\def\W{\mathbb W}
	\def\A{\scr A} \def\Lip{{\rm Lip}}\def\S{\mathbb S}
	\def\BB{\mathbb B}\def\Ent{{\rm Ent}} \def\i{{\rm i}}\def\itparallel{{\it\parallel}}
	\def\g{{\mathbf g}}\def\Sect{{\mathcal Sec}}\def\T{\mathcal T}\def\V{{\bf V}}
	\def\PP{{\bf P}}\def\HL{{\bf L}}\def\Id{{\rm Id}}\def\f{{\bf f}}\def\cut{{\rm cut}}
	\def\Ss{\mathbb S}
	\def\BL{\scr A}\def\Pp{\mathbb P}\def\Pp{\mathbb P} \def\Ee{\mathbb E} \def\8{\infty}\def\1{{\bf 1}}
	\maketitle
	\def\Cc{\mathcal C} \def\t{\theta}
	\begin{abstract}	We consider stochastic different equations on $\R^d$ with coefficients depending on the path and distribution for the whole history. Under a local integrability condition on the   time-spatial singular drift, the well-posedness and Lipschitz continuity in initial values are proved, which is new even in the distribution independent case. Moreover, under a   monotone condition,   the asymptotic log-Harnack inequality is established, which extends the corresponding result of \cite{BWY19} derived in the distribution independent case.
	\end{abstract}
	
	\section{Introduction}
The dimension-free Harnack inequality with power was first introduced in \cite{W97} to study the log-Sobolev inequality
on Riemanian manifolds, and has been intensively extended and applied to derive regularity estimates many for SDEs (stochastic differential equations), SPDEs (stochastic partial differential equations), path dependent SDEs, and distribution-dependent SDEs, see \cite{Wbook, RW24, R21} and references therein.
As a limit version when the power goes to infinite, the log-Harnack inequality has been introduced in
\cite{W10} to characterize the curvature lower bounded and entropy-cost estimates, and has been extended
to metric measure spaces \cite{Am, KS}. When the noise of a stochastic system is weak  such that
the log-Harnack inequality is not available, the asymptotic log-Harnack inequality
has been studied in \cite{BWY19, Xu} for path-dependent SDEs and SPDEs, which in particular implies an asymptotic gradient estimate.

In this paper, we study the well-posedness and asymptotic log-Harnack inequality   for path-distribution dependent SDEs with infinite memory.

Let $(\R^d, |\cdot |)$ be the $d$-dimensional Euclidean space for some $d\in\mathbb N$.
Denote by $\R^d\otimes\R^d$   the family of all $d\times d$-matrices
with real entries, which is  equipped with the operator norm $\|\cdot\|$.    Let  $A^*$ denote the   transpose of $A\in \R^d\otimes \R^d$, and let
$\|\cdot\|_{\infty}$ be the uniform norm for functions taking values in $\R,\R^d$ or $\R^d\otimes \R^d$.

To describe the path dependence with exponential decay  memory, {let $\C:=C((-\8,0];\R^d)$ and} for  $\tau>0$, set
\begin{align}\label{CR}
	\C_\tau=\left\{\xi\in\C:\|\xi\|_\tau:=\sup_{s\in(-\infty,0]}(\e^{\tau s}|\xi(s)|)<\8\right\}.
\end{align}
It is well known that  $(C((-\infty,0];\R^d),\|\cdot\|_\infty)$  is complete but not separable, so is   $(\C_\tau, \|\cdot\|_\tau)$ due to the  isometric  
	$$\C_\tau\ni \xi:=(\xi_s)_{s\in (-\infty,0]}\mapsto \e^{\tau\cdot}\xi:= (\e^{\tau s}\xi_s)_{s\in (-\infty,0]}\in C((-\infty,0];\R^d).$$

Let $\scr P$ and $\scr P(\C_\tau)$ be the set of all probability measures on $(\R^d,\B(\R^d))$ and $(\C_\tau,\B(\C_\tau))$, respectively, equipped with the weak topology. Let $\B_b(\C_\tau)$ be the class of bounded measurable functions on $\C_\tau$, and $\B_b^+(\C_\tau)$ the set of strictly positive functions in $\B_b(\C_\tau)$.

Let $(W(t))_{t\geq0}$ be  a $d$-dimensional Brownian motion defined on a complete filtered probability space $(\Omega,\F,(\F_t)_{t\geq0},\P)$.	 For an $\F_0$-measurable random variable $X_0:=((-\infty,0]\ni r\mapsto X(r))$   on $\C_\tau$,
we consider the following path-distribution dependent SDE with infinite memory:
\begin{align}\label{E-1}
\d X(t)=b(t,X_t,\L_{X_t})\d t+\si(t,X_t)\mathrm{d} W(t), \quad t\ge 0,
\end{align} 
where
for each fixed $t\geq 0$, $X_t(\cdot)\in\C_\tau$ is  defined by
\[X_t(r):=X(t+r),\ \ r\in(-\infty,0],\]
which is called the segment process of $X(t)$, $\L_{X_t}\in \scr P(\C_\tau)$ is the distribution of $X_t$, and
$$b:\R^+\times\C_\tau\times\scr P(\C_\tau)\to\R^d,\ \ \
\si:\R^+\times\C_\tau\to\R^d\otimes\R^d $$ are measurable. When different probability spaces are concerned,
we use $\L_{X_t|\P}$ in place of  $\L_{X_t}$ to emphasize the underline probability.

For any constant  $k\geq 0$,   let
$$\scr P_k(\C_\tau):=\big\{\mu\in\scr P(\C_\tau):\|\mu\|_k:= \mu(\|\cdot\|_{\tau}^k)^{\ff 1 k}<\8\big\},$$
where for $k=0$ we set $\mu(\|\cdot\|_{\tau}^k)^{\ff 1 k}=1$ such that $\scr P_0(\C_\tau)= \scr P (\C_\tau).$
When $k>0$, $\scr P_k(\C_\tau)$
is a complete metric space under the $L^k$-Wasserstein distance,
$$\W_k(\mu,\nu):=\sup_{N\ge 1} \inf_{\pi\in C(\mu,\nu)}\bigg(\int_{\C_\tau\times\C_\tau}\|\xi-\eta\|_{N,\tau}^k\,\pi(\mathrm{d}\xi,\mathrm{d}\eta)\bigg)^{\ff{1}{1\vee k}},\ \ \mu,\nu\in\scr P_k(\C_\tau),$$
where $C(\mu,\nu)$ is the set of all coupling of $\mu$ and $\nu$, and
$$\|\xi\|_{N,\tau}:=\sup_{s\in [-N,0]}(\e^{\tau s}|\xi(s)|),\ \ \ N\in \mathbb N.$$
To see this, let $\scr P_k(\C_{N,\tau})$ be the space of all probability measures on  $\C_{N,\tau}:=C([-N,0];\R^d)$ with finite $k$-moment of the uniform norm, let $ \mu_N\in \scr P_k(\C_{N,\tau})$ be  the marginal 
distribution on $\C_{N,\tau}$ of $\mu$, and let 
$$\W_k(\mu_N,\nu_N):=  \inf_{\pi\in C(\mu_N,\nu_N)}\bigg(\int_{\C_{N,\tau}\times\C_{N,\tau}}\|\xi-\eta\|_{N,\tau}^k\,\pi(\mathrm{d}\xi,\mathrm{d}\eta)\bigg)^{\ff{1}{1\vee k}}.$$
Since $\C_{N,\tau}$ is Polish under the norm $\|\cdot\|_{N,\tau}$, $(\scr P_k(\C_{N,\tau}), \W_k)$ is a Polish space as well, and 
\beq\label{EQ} \W_k(\mu_N,\nu_N)= \tt\W_k(\mu_N,\nu_N):=\inf_{\pi\in C(\mu,\nu)}\bigg(\int_{\C_{\tau}\times\C_{\tau}}\|\xi-\eta\|_{N,\tau}^k\,\pi(\mathrm{d}\xi,\mathrm{d}\eta)\bigg)^{\ff{1}{1\vee k}}.\end{equation} 
Since the marginal distribution on $\C_{N,\tau}\times\C_{N,\tau}$ of a coupling for $\mu$ and $\nu$ is a coupling of $\mu_N$ and $\nu_N$, we have
$$\W_k(\mu_N,\nu_N)\le \tt\W_k(\mu_N,\nu_N).$$ On the other hand, let 
$\pi_N\in \C(\mu_N,\nu_N)$ such that
$$\W_k(\mu_N,\nu_N)= \bigg(\int_{\C_{N,\tau}\times\C_{N,\tau}}\|\xi_N-\eta_N\|_{N,\tau}^k\,\pi_N(\mathrm{d}\xi_N,\mathrm{d}\eta_N)\bigg)^{\ff{1}{1\vee k}}.$$
Noting that $\C_\tau$ is separable under norm $\|\cdot\|_{\tau+1}$,   the completeness $\overline{\C_{\tau}}$ of $\C_\tau$ under this norm becomes a Polish space.
Since
$$\|\cdot\|_{\tau}=\lim_{N\to\infty}\|\cdot\|_{N,\tau},\ \ \ \tau\ge 0,$$
and $\|\cdot\|_{N,\tau}$ is continuous with respect to $\|\cdot\|_{\tau'}$ for any $\tau,\tau'\ge 0$,  we conclude that $\|\cdot\|_\tau$ and $\|\cdot\|_{\tau+1}$ induce the same Borel $\si$-field on $\overline{\C_{\tau}}$. 
So,  by extending $\mu,\nu\in\scr P_k(\C_{\tau})$ 
as probability measures on the Polish space $\overline{\C_{\tau}}$ such that $\overline{\C_{\tau}}\setminus \C_\tau$ is a null set,  $\mu$ and $\nu$   have regular conditional distributions  
$ \mu(\cdot|\xi_N)$ and $ \nu(\cdot |\xi_N)$ 
 on $C((-\infty,-N);\R^d)$ given $\xi_N\in \C_{N,\tau}$. So, for any $\xi\in \C_\tau$ letting
 $$\xi_N:= \xi|_{[-N,0]},\ \ \ \xi_N^c:= \xi|_{(-\infty,-N)},$$ the measure 
 $$\pi(\d \xi,\d\eta):= \pi_N(\d\xi_N,\d\eta_N) \mu(\d \xi_N^c|\xi_N)\nu(\d\eta_N^c|\eta_N)$$\
 is a coupling of $\mu$ and $\nu$, and
\beg{align*} &\tt\W_k(\mu_N,\nu_N)\le   \bigg(\int_{\C_{\tau}\times\C_{\tau}}\|\xi-\eta\|_{N,\tau}^k\,\pi(\mathrm{d}\xi,\mathrm{d}\eta)\bigg)^{\ff{1}{1\vee k}}\\
&=\bigg(\int_{\C_{\tau}\times\C_{\tau}}\|\xi-\eta\|_{N,\tau}^k\,\pi_N(\mathrm{d}\xi,\mathrm{d}\eta)\bigg)^{\ff{1}{1\vee k}}= \W_k(\mu_N,\nu_N).\end{align*} 
 Thus, \eqref{EQ} holds, so that   $\W_k$  is a complete metric on $\scr P_k(\C_{N,\tau})$, which trivially implies the triangle inequality of $\W_k$ on $\scr P_k(\C_\tau)$.  
 If $\{\mu^{(n)}\}_{n\ge 1}$ is a $\W_k$-Cauchy sequence in $\scr P_k(\C_\tau)$,
 then so is $\{\mu_N^{(n)}\}_{n\ge 1}$ for every $N\in\mathbb N$. Hence,  $\mu_N^{(n)}$ has a unique limit $\mu_N$ in $\scr P_k(\C_{N,\tau})$ under $\W_k$, and 
 the family $\{\mu_N\}_{N\ge 1}$ is consistent, so that by Kolmogorov's extension theorem, there exists a unique $\mu\in \scr P_k(\C_{\tau})$ with $\{\mu_N\}_{N\ge 1}$  as marginal distributions, so that
\beg{align*}&\lim_{n\to\infty} \W_k(\mu^{(n)},\mu)= \lim_{n\to\infty} \sup_{N\ge 1} \W_k(\mu_N^{(n)},\mu_N) = \lim_{n\to\infty}  \sup_{N\ge 1} \lim_{m\to\infty} \W_k(\mu_N^{(n)},\mu_N^{(m)}) \\
&\le \lim_{n\to\infty}   \lim_{m\to\infty}  \sup_{N\ge 1} \W_k(\mu_N^{(n)},\mu_N^{(m)}) = \lim_{n\to\infty}   \lim_{m\to\infty}  \W_k(\mu^{(n)},\mu^{(m)}) =0.\end{align*}
Hence, $\scr P_k(\C_\tau)$ is complete under $\W_k$.

Moreover, for any $k\ge 0$, $\scr P_k(\C_\tau)$ is a complete  metric space under the weighted variation distance		$$
\|\mu-\nu\|_{k,var}:=\sup_{f\in\B_b(\C_\tau),|f|\leq 1+\|\cdot\|_\tau^k}|\mu(f)-\nu(f)|=|\mu-\nu|(1+\|\cdot\|_\tau^k),
$$		
where	$|\mu-\nu|$ is the total variation of $\mu-\nu$.  According to \cite[Remark 3.2.1]{RW24}, for any $k>0$, there exists a constant $c>0$ such that
\begin{align}\label{vark}
\|\mu-\nu\|_{var}+\W_k(\mu,\nu)^{1\vee k}\leq c\|\mu-\nu\|_{k,var}, \ \ \mu,\nu\in\scr P_k(\C_\tau).
\end{align}Denote $\W_{k,var}=\W_k+\|\cdot\|_{k, var}$ for simplicity.

We will solve \eqref{E-1} for distribution   $\L_{X_t}$ belonging to
a subclass  $\tt{\scr P}_k(\C_\tau)\subset\scr P_k(\C_\tau)$ such that $\L_{X_t}$ is weakly continuous in $t\ge 0$.

\beg{defn} \beg{enumerate}	\item[(1)] {An adapted continuous process  $(X_t)_{t\ge 0}$  on $\C_\tau$ is called a  segment solution of \eqref{E-1} with initial value $X_0$, if $X_0$ is an $\F_0$-measurable random variable on $\C_{\tau}$,
and $(X(t):=X_t(0))_{t\ge 0}$ satisfies }
$$\int_0^t\E\big[|b(r,X_{r},\L_{X_{r}})|+  \| \si(r,X_r)\|^2\big|\F_0\big] \,\mathrm{d} r<\infty,\ \ t\ge 0,$$ and  $\P$-a.s.
$$X(t) = X(0) +\int_0^{t} b(r,X_r, \L_{X_{r}})\,\mathrm{d} r+ \int_0^{t} \si(r,X_r)\,\mathrm{d} W(r),\ \ t\ge 0.$$  In this case $(X(t))_{t\ge 0}$ is called the (strong) solution.
The SDE \eqref{E-1} is called  strongly   well-posed for   distributions in $\tt{\scr P}_k(\C_\tau)$, if for any $\F_0$-measurable $X_0$  with $\L_{X_0}\in \tt{\scr P}_k(\C_\tau)$, it has a unique segment solution with $\L_{X_t}\in  \tt{\scr P}_k(\C_\tau) $ for $t\ge 0$.
\item[(2)] A couple   $(X_t, W(t))_{t\ge 0}$   is called a weak segment solution of \eqref{E-1} with initial distribution $\mu\in \tt{\scr P_k}(\C_\tau),$  if there exists a probability space under which $W(t)$ is $d$-dimensional Brownian motion and $\L_{X_0}=\mu$, such that $(X(t):= X_t(0))_{t\ge 0}$
solves the SDE \eqref{E-1}.
We call \eqref{E-1} weakly unique, if for any two weak segment solutions  $(X_t^i, W^i(t))_{t\ge 0}$
under probabilities $\P^i$ with common initial distribution, we have $\L_{X_t^1|\P^1}=\L_{X_t^2|\P^2}$ for all $t\ge 0$. We call \eqref{E-1} weakly well-posed for distributions in $\tt{\scr P}_k(\C_\tau)$, if for any initial distribution it has a unique  weak segment solution.
\item[(3)]	The SDE \eqref{E-1}  is called well-posed  for distributions in $\tt{\scr P}_k(\C_\tau)$,  if it is both  strongly and weakly well-posed for distributions in $\tt{\scr P}_k(\C_\tau)$. In this case,  for any $\xi\in\C_\tau$ such that $\L_{\xi}\in\tt{\scr P}_k(\C_\tau)$,  let  $$P_t^*\gamma=\L_{X_t^\xi},\ \ \gamma=\L_{\xi}$$  and
denote
\begin{align*}
	P_tf(\gamma):=\E[f(X_t^\xi)]=\int_{\C_\tau}f\,\d\left\{P_t^*\gamma\right\},\ \ \gamma=\L_{\xi},\, t\geq0,\,f\in\B_b(\C_\tau),
\end{align*}
where $(X_t^\xi)_{t\geq0}$ is the unique solution to \eqref{E-1} with $X_0^\xi=\xi$.\end{enumerate}
\end{defn}

To characterize  the singularity of coefficients in time--space variables, we recall some functional spaces introduced in \cite{XXZZ}.  For any $p\ge 1$, $L^p(\R^d)$ is the class of   measurable  functions $f$ on $\R^d$ such that
$$\|f\|_{L^p(\R^d)}:=\bigg(\int_{\R^d}|f(x)|^p\,\mathrm{d} x\bigg)^{\ff 1 p}<\infty.$$
For any $p,q\geq1$ and $0\leq s<t$, let $\tt L_q^p(s,t)$ denote the class of measurable functions $f$ on $[s, t]\times\R^d$ such that
\begin{align}\label{LPQ}
\|f\|_{\tt L_q^p(s,t)}:= \sup_{z\in \R^d}\bigg( \int_{s}^{t} \|1_{B(z,1)}f_r\|_{L^p(\R^d)}^q\,\mathrm{d} r\bigg)^{\ff 1 q}<\infty,
\end{align}
where $B(z,1):= \{x\in\R^d: |x-z|\le 1\}$.

When $s=0,$ we simply   denote $$\tt L_q^p(t)= \tt L_q^p(0,t),\ \ t>0.$$
We will take $(p,q)$ from the class
\begin{align}\label{KK}
\scr K:=\Big\{(p,q): p,q\in  ({2},\infty),\  \ff d p+\ff 2 q<1\Big\}.
\end{align}

\section{The Singular Case: Well-Posedness and Lipschitz Continuity in Initial Value}
In  this section, we let $k\geq0$ and  consider \eqref{E-1} with singular drifts and  $\si(t,\xi)=\si(t,\xi(0))$.  
\subsection{Path Dependent SDEs with Infinite Memory}
In this part, we consider the following path dependent SDE with infinite memory on $\R^d$:
\beq\label{E-2}
\d X(t)= b(t,X_t)\mathrm{d} t +\si(t,X(t))\mathrm{d} W(t),\ \  t>0,\ \ X_0=\xi\in\C_\tau.
\end{equation}

To ensure the existence and uniqueness of solutions to \eqref{E-2}, we decompose $b$ as
$$b(t,\xi)= b^{(0)}(t,\xi(0))+ b^{(1)}(t,\xi),\ \ t\ge 0,\ \xi\in\C_\tau$$ and make the following assumptions
on $b^{(0)}, b^{(1)}$ and $\si$. For any $\xi\in\C_\tau$, let
$\xi^0\in \C_\tau$ be defined as
$$\xi^0(r)=\xi(0),\ \ r\le 0.$$

\beg{enumerate}
\item[($A_1$)]   $a:=\si\si^*$ is invertible with $\| a\|_\8+\| (a)^{-1}\|_\8<K$ for some constant $K>0$ and
$$		\lim_{\vv\downarrow 0}\sup_{|x-y|\leq\vv,t\in[0,T]}{\| a_t(x)-a_t(y)\|}=0, \quad T\in (0,\infty),\  x,y\in\R^d.	$$
\item[($A_2$)]
There exist constants $\{(p_i,q_i)\}_{0\leq i\leq l}\in\scr K$ with $l\geq 1$, $p_i> 2$, and functions $0\leq f_i\in \cap_{n\in\mathbb N}\tilde{L}^{p_i}_{q_i}(n)$, $0\leq i\leq l$, such that
$$|b^{(0)}|\leq f_0,\ \ \|\nn\si\|\leq\sum_{i=0}^{l}{f_i}.$$
\item[$(A_3)$]
For every $n>0$, there exists a constant $K_n>0$ such that
\begin{equation}\label{BL2}
|b^{(1)}(t,\xi)-b^{(1)}(t,\eta)|\leq K_n\|\xi-\eta\|_\tau, \ \ {t\geq0}, \|\xi\|_\tau,\|\eta\|_\tau\leq n.
\end{equation} Moreover, there exists a constant $K>0$ such that
\begin{equation}
|b^{(1)}(t,\xi)-  b^{(1)}(t,\xi^0)|\leq K(1+\|\xi\|_\tau),\ \ {t\geq0}, \xi\in\C_\tau.
\end{equation}
\end{enumerate}

\begin{thm}\label{TB1}
Assume $(A_1)$--$(A_3)$. Then \eqref{E-2} is well-posed for any initial value in $\C_\tau$, and for any constants $k,T>0 $, there exists a constant  $c >0$, such that
\begin{align}\label{EV0}			\E\left[\sup_{t\in[0,T]}\left(1+ \|X_t\|_\tau^k\right)\bigg|X_0 \right]\leq c \left(1+ \|X_0\|^k_\tau\right).	\end{align}
\end{thm}

Let $X_t^\xi$ be the segment solution with $X_0^\xi=\xi$. To ensure the Lipschitz continuity of
$X_t^\xi$ in $\xi$, we
strengthen  $(A_3)$   to

\begin{itemize}
\item[$(A_3')$]
 $\sup_{t\geq 0}|b^{(1)}(t,{\bf 0})|<\8$, and  there exist constants   $K>0$ and $\aa\in [0,1]$ such that 
\begin{align*}
&|b^{(1)}(t,\xi)-b^{(1)}(t,\eta)|\leq K\|\xi-\eta\|_\tau,\\
&	|b^{(1)}(t,\xi)-   b^{(1)}(t,\xi^0)|\leq K\left(1+\|\xi\|_\tau^{\alpha}\right), \ \ t\geq 0,\ \xi,\eta\in\C_\tau.
\end{align*}
\end{itemize}

\beg{thm}\label{TN} Assume $(A_1)$, $(A_2)$ and $(A_3')$.
Then for any constants $\vv\in (0,1)$ and $k,T\geq 1$, there exists a constant $c >0$ such that
\begin{align}\label{E01}
\E\left[\sup_{t\in[0,T]}\|X^\xi_t-X_t^\eta\|_\tau^k\right]\leq c \e^{\vv\left(\|\xi\|_\tau^{2\alpha}+\|\eta\|_\tau^{2\alpha}\right)}\|\xi-\eta\|_\tau^k,\ \ \xi,\eta\in\C_\tau.
\end{align}
\end{thm}

\subsection{Path-Distribution Dependent SDEs with Infinite Memory}

To ensure the well-posedness of solutions  to \eqref{E-1} with singular drift for distributions in
$\scr P_k(\C_\tau)$,  we decompose $b$ as
$$b(t,\xi,\mu)= b^{(0)}(t,\xi(0))+ b^{(1)}(t,\xi,\mu),\ \ t\ge 0,\ \xi\in \C_\tau,\ \mu\in \scr P_k(\C_\tau),$$ and  impose the following assumptions on
$b^{(0)}, b^{(1)}$  and $\si$. 
\beg{enumerate}
\item[$(H_1)$]
$b^{(0)}$ and $a:=\si\si^*$ satisfy $(A_1)$ and $(A_2)$.
\item[($H_2$)]
 $\sup_{t\geq0}\{|b(t,{\bf 0},\delta_{\bf 0})|+\|\sigma(t,{\bf 0})\|\}<\infty$, and  there exists a function $H\in L^1_{loc}([0,\infty);(0,\8))$ and  constants {$K >0$}, $\alpha\in[0,1]$ such that
\begin{align}\label{blip}
&  |b^{(1)}(t,\xi,\mu)-b^{(1)}(t,\eta,\nu)|^2\leq {K}\|\xi-\eta\|_\tau^2+H(t)\W_{k,var}(\mu,\nu)^2,\\
&  |b^{(1)}(t,\xi,\mu){-b^{(1)}(t,\xi^0,\mu)}|\leq {K}\left\{1+\|\xi\|_\tau^\alpha+\|\mu\|_k\right\},\notag\\
&\qquad \ \ t\ge 0,\ \xi,\eta\in\C_\tau,\ \mu,\nu\in\scr P(\C_\tau).
\end{align}

\end{enumerate}

\begin{thm}\label{T01}
Assume {$(H_1)$-$(H_2)$}. Then the SDE $\eqref{E-1}$ is well-posed for distributions in $\scr P^\alpha_{k,\e}(\C_\tau)$, where
\begin{align}\label{P-a}
\scr P^\alpha_{k,\e}(\C_\tau):=\left\{\mu\in\scr P_k(\C_\tau):\mu\left(\e^{\vv\|\cdot\|_\tau^{2\alpha}}\right)<\8\text{ for some }\vv\in(0,1)\right\}.
\end{align}Moreover, for any $n\ge 1$ and $T\in (0,\infty)$,   there exists a constant $c >0$ such that any  solution to $\eqref{E-1}$ satisfies
\begin{align}\label{xn}
\E\left[\sup_{t\in[0,T]}\|X_t\|_\tau^n\bigg| X_0\right]\leq c \left(1+\|X_0\|_\tau^n\right).
\end{align}
\end{thm}
\section{The Monotone Case: Well-Posedness and Asymptotic log-Harnack inequality} In  this section, we  consider \eqref{E-1} with monotone coefficients and 
establish the well-posedness and asymptotic log-Harnack inequality. 
\subsection{Path Dependent SDEs with Infinite Memory}

Note that when the SDE \begin{equation}\label{E3}
\d X(t)= b(t,X_t)\mathrm{d} t +\si(t,X_t)\mathrm{d} W(t),\ \  t>0,\ \ X_0=\xi\in\C_\tau.
\end{equation} is time-homogenous and monotone,  the asymptotic log-Harnack inequality has been derived in \cite[Theorem 3.1]{BWY19} for the Markov semigroup
$$
P_tf(\xi):=\E \big[f(X_t^\xi)\big], \quad\ \  t \ge 0, \    \xi\in \C_\tau, \  f\in \B_b(\C_\tau).
$$
We  call $\mu$ an invariant probability measure of $P_t$, if it is a probability measure on $\C_\tau$ such that
$$\int_{\C_\tau} P_t f\,\d\mu= \int_{\C_\tau}f\,\d\mu,\ \ f\in \B_b(\C_\tau).$$
An important application of this inequality is the gradient estimate.
For any function $f$ on $\C_\tau$, let
$$\|\nn f(\xi)\|_\infty:=\sup_{\eta\in \C_\tau\setminus\{\xi\}}\ff{|f(\xi)-f(\eta)|}{\|\xi-\eta\|_\tau}$$
be the Lipschitz constant of a function $f$ at point $\xi\in\C_\tau$, which coincides with the   norm of the gradient   $\nn f(\xi)$ if $f$ is G\^ateaux differentiable at $\xi$. We denote $f\in C_{b,L}(\C_\tau)$ if $$\|f\|_\infty+\|\nn f\|_\infty<\infty.$$
Since the proof of \cite[Theorem 3.1 and Theorem A.1]{BWY19} apply also to the time-inhomogenous case which is crucial in the study of distribution dependent SDEs, we reformulate this result and its consequence  for \eqref{E3} with time dependent coefficients without proof.

\beg{enumerate}\item[$(B)$] $b(t,\cdot)\in C(\C_\tau{;\R^d})$ for each $t\ge 0$, $\si$ is invertible, and there exists a constant $K>0$ such that 
\beg{align*}
&\|\si(t,\xi)-\si(t,\eta)\|^2+\<b(t,\xi)-b(t,\eta), \xi(0)-\eta(0)\>^+\le K\|\xi-\eta\|_\tau^2,\\
&\|\si(t,\xi)\|+\|\si(t,\xi)^{-1}\|+|b(t,{\bf 0})|\le K,\ \ \ t\ge 0,\ \xi,\eta\in \C_\tau,\end{align*}
where ${\bf 0}\in \C_\tau$ with ${\bf 0}(r)=0, r\le 0.$
\end{enumerate}

\begin{thm}\label{TB2} Under assumption $(B)$, for any $\xi\in\C_\tau$, the SDE $\eqref{E3}$
has a unique solution with $X_0=\xi$ such that $X_t^\xi$ is a Markov process on $\C_\tau$.
Moreover,  for any $\tau_0\in(0,\tau)$, there exists a constant $c>0$ such that for any $t\ge 0$, $\xi,\eta\in\C_\tau$ and $f\in\B^+_b(\C_\tau)$ { with $\|\nn\log f\|_\8<\8$},
\begin{align}\label{AL1}& \E\big[\|X_t^\xi\|_\tau^2\big]\le c\e^{ct} \big(1+\|\xi\|_\tau^2\big),\notag\\
&P_t\log f(\eta)\leq \log P_tf(\xi)+
c\|\xi-\eta\|_\tau^2+c\e^{-\tau_0t}\|\nn\log f\|_\8 \|\xi-\eta\|_\tau.
\end{align}

\end{thm}

The following result is a direct consequence of  Theorem \ref{TB2} and \cite[Theorem 2.1]{BWY19} with
\beq\label{GM}  E=\C_\tau,\ \ \ \rho(\xi,\eta)=\|\xi-\eta\|_\tau,\ \ \  \Gamma_t(\xi)=c\e^{  -\tau_0t},\ \ \ \LL(\xi)=c.\end{equation}

\beg{cor} In the situation of Theorem $\ref{TB2}$,
the following assertions hold.
\beg{enumerate} 	\item[$(1)$]  For any $t\ge 0,\xi\in \C_\tau$ and $f\in  C_{b,L}(\C_\tau)$,
$$
|\nn P_t f(\xi)|  \leq \ss{2c[P_t f^2-(P_tf)^2)]} (\xi)  +c\e^{-\tau_0 t} \|\nn f\|_\infty.$$
\item[$(2)$] When the coefficients do not depend on $t$,
$P_t$ has at most one invariant probability measure, and if   $\mu$ is its invariant probability measure,
then
\begin{align*}
\limsup_{t\to\8}P_t f(\xi)\leq\log\left(\ff{\mu(\e^{f})}{\int_{\C_\tau}\e^{-c\|\xi-\eta\|_\tau^2}\mu(\d \eta)} \right), \ \ \xi\in\C_\tau,\ f\in \B_b^+(\C_\tau).
\end{align*}
Consequently, for any closed set $A\subset\C_\tau$ with $\mu(A)=0$,
\begin{align*}
\lim_{t\to\8}P_t\1_A(\xi)=0.\ \ \xi\in\C_\tau.
\end{align*}
\item[$(3)$]  Let $\xi\in\C_\tau$ and $A\subset\C_\tau$ be a measurable set such that
\begin{align*}
\delta(\xi,A):= \liminf_{t\to\8}P_t(\xi,A)>0.
\end{align*}
Then for any $\vv>0$ and $A_\vv:=\{\eta\in\C_\tau:\|\eta-\xi\|<\vv,\xi\in A\}$,
\begin{align*}
\liminf_{t\to\8}P_t(\eta,A_\vv)>0,\ \ \eta\in\C_\tau.
\end{align*}
Moreover precisely,  for any $\vv_0\in(0,\delta(\xi,A))$, there exists a constant $t_0>0$ such that for any
$\eta\in \C_\tau$ and $\vv>0$,
\begin{align*}
\inf\bigg\{P_t(\eta, A_\vv):\ t>t_0\lor \Big(\ff 1 {\tau_0}\log\ff{\LL(\eta)\|\xi-\eta\|_\tau}{\vv \vv_0}\Big)\bigg\}>0,\ \ \eta\in \C_\tau.
\end{align*}\end{enumerate}
\end{cor}
\subsection{Path-Distribution Dependent SDEs with Infinite Memory}
we assume that the following monotone assumption holds,
as in \cite{BWY19} in the distribution-free case.

\beg{enumerate} \item[$(H')$] $\sup_{t\ge 0}|b(t,{\bf 0},\dd_{\bf 0})|<\8$, and 
$b(t,\cdot,\mu)\in C(\C_\tau;\R^d)$ holds for any $t\ge 0$, $\mu\in \scr P(\C_\tau)$. Moreover,   there exist constants $K_1, K_2>0$ such that
and 
\beg{align*}&\<b(t,\xi,\mu)-b(t,\eta,\nu),\xi(0)-\eta(0)\>^++\|\si(t,\xi)-\si(t,\eta)\|^2\\
&\le K_1\|\xi-\eta\|_\tau^2+K_2\W_2(\mu,\nu)^2,
\ \ t\ge 0,\ \xi,\eta\in\C_\tau,\ \mu,\nu\in \scr P_2(\C_\tau).\end{align*}
\end{enumerate}
\begin{thm}\label{T02}
Assume $(H')$. Then the following assertions hold.
\beg{enumerate}\item[$(1)$]  The SDE \eqref{E-1} is well-posed for distributions in $\scr P_2(\C_\tau)$, 
and there exists a constant $c>0$ such that
\beq\label{*N} \W_2(P_t^*\mu_0,P_t^*\nu_0)\le c\e^{ct} \W_2(\mu_0,\nu_0),\ \ t\ge 0,\ \mu_0,\nu_0\in \scr P_2(\C_\tau).\end{equation}
\item[$(2)$] If $\si$ is invertible with $\|\si\|_\infty+\|\si^{-1}\|_\infty<\infty$, then 
for any $\tau_0\in(0,\tau)$, there exists a  constant $c\ge 1$ such that
\begin{equation}\label{log1}
\begin{aligned}
P_t\log f(\nu)\leq & \log P_tf(\mu)+
{ cK_2\e^{ct}}\W_{2}(\mu,\nu)^2+c\e^{-\tau_0t}\|\nn\log f\|_\8\W_{2}(\mu,\nu)
\end{aligned}
\end{equation}
holds for $t\geq0$, $\mu,\nu\in\scr P_{2}(\C_\tau)$ and $f\in\B^+_b(\C_\tau)$ {with $\|\nn\log f\|_\8<\8$}.   
Consequently, for any $t>0$ and $f\in C_{b,L}(\C_\tau)$,
\begin{align}\label{lip}
|P_tf(\mu)-P_tf(\nu)|\le \W_{2}(\mu,\nu) \big[ c\e^{-\tau_0t}\|\nn f\|_\8+ {2\ss{cK_2\e^{ct}}}\|f\|_\8\big].
\end{align}\end{enumerate}
\end{thm}
 \begin{rem} Note that 
\eqref{log1} implies  \eqref{AL1} for $\mu=\dd_{\xi}$, $\nu=\dd_{\eta}$ and  $K_2=0$.
\end{rem}

\section{Proof of Theorem \ref{TB1} Theorem \ref{TN}}

To prove  these two theorems , we present some lemmas below.

\smallskip
\begin{lem}\label{LB1}
Assume $(A_1)$, $(A_2)$,  $(A_3')$  and $\|b^{(1)}\|_\8<\8$.
Then for any initial value $\xi\in\C_\tau$, \eqref{E-2} has a unique non-explosive strong solution satisfying \eqref{EV0}, and for any $T\in (0,\infty)$ there exists a constant $c>0$ such that
\begin{align}\label{SUPE1}
{\E\left[\sup_{t\in[0,T]} \big|X^{\xi}(t)-X^{\eta}(t)\big|\right]\le c\|\xi-\eta\|_\tau}, \ \ \ \xi,  \eta\in\C_\tau.
\end{align}
\end{lem}\beg{proof} The desired estimate follows from the proof
of \cite[Lemma 4.1]{Z24}  with $\|\cdot\|_\tau$ in place of uniform norm $\|\cdot\|_\C$ on the path space with finite time interval.  \end{proof}

Next, consider the local Hardy-Littlewood maximal function for a nonnegative function $f$ on $\R^d$:
\begin{align}\label{max}
\M f(x):=\sup_{r\in(0,1)}\frac{1}{|B(0,r)|}\int_{B(0,r)}f(x+y)\,\mathrm{d} y,\ \   x\in \R^d,
\end{align}
where $|B(0,r)|$ is the volume of $B(0,r):=\{y: |y|<r\}$.
The following result is taken from \cite[Lemma 2.1]{XZ20}.
\smallskip
\begin{lem} \label{Hardy}
\beg{enumerate}\item[$(1)$] There exists a constant $c>0$ such that for any   $f\in C_b(\R^d)$ with $|\nn f|\in L^1_{loc}(\R^d)$,
\beq\label{HH1}
|f(x)-f(y)|\leq c|x-y|(\M |\nabla f|(x)+\M |\nabla f|(y)+\|f\|_\8),\ \    x,y\in\R^d.\end{equation}
\item[$(2)$] For any     $T,p,q\in (1,\infty),  $ there exists a constant $c>0$ such that
\beq\label{HH2}
\|\M f\|_{\tilde{L}_q^p(T)}\leq c\|f\|_{\tilde{L}_q^p(T)},\ \ f\in \tilde{L}_q^p(T).
\end{equation}\end{enumerate}
\end{lem}

The lemma below examines the exponential integrability of functionals for segment process.

\begin{lem}\label{EXPB}
Assume $(A_1)$, $(A_2)$ and $(A_3')$. For any $T\in (1,\infty)$,
there exist  constants $\beta_T,k_T >0$
such that  for any   solution $(X_t)_{t\in [0,T]}$ to \eqref{E-2},
\begin{align}\label{EXPB1}
\E\left[\exp\left(\beta\int_{s}^{t}\| X_u\|_\tau^{2\alpha}\,\mathrm{d} u\right)\bigg|\F_{s}\right]\leq \e^{k_T\bb \left(1+\|X_s\|_\tau^{2\alpha}\right)}
\end{align}
holds for any $0\leq\bb\leq\bb_T$ and $0\leq s\leq t\leq T$.
\end{lem}
\begin{proof}
Obviously, \eqref{EXPB1} holds for $\alpha=0$. By shifting the starting time from $0$ to $s$,  we may and do assume that $s=0$. So, by Jensen's inequality,  it suffices to find constants $\bb_T,k_T>0$ such that
\begin{align}\label{exp0}
\E\left[\exp\left(\beta_T\int_{0}^{T}\| X_u\|_\tau^{2\alpha}\,\mathrm{d} u\right)\bigg|\F_{0}\right]\leq \e^{k_T{\beta_T}\left(1+\|X_0\|_\tau^{2\alpha}\right)}.
\end{align}  By \cite[Theorem 3.2]{XXZZ} and \cite[Theorem 2.1]{ZY}, there exists a constant $\lambda_0>0$ such that for any $\lambda\geq\lambda_0$, the PDE for $u(t,\cdot):\R^d\to\R^d$
\begin{align}\label{PDE1}
(\partial_t+L^0(t))u(t,\cdot)=\lambda u(t,\cdot)-b^{(0)}(t,\cdot),\ \ t\in[0,T],\ \ u_T=0
\end{align}
has a unique solution in $\tilde{H}_{q_0}^{2,p_0}(T)$, where $L^0(t):=\nn_{b^{(0)}(t,\cdot)}+\ff{1}{2}\text{tr}(\{\si\si^*\}(t,\cdot)\nn^2)$, and there exist constants $c,\theta>0$ such that
\begin{align}\label{LAMU}
\lambda^\theta(\| u\|_\8+\|\nn u\|_\8)+\|\partial_tu\|_{\tilde{L}_{q_0}^{p_0}(T)}+\|\nn^2u\|_{\tilde{L}_{q_0}^{p_0}(T)}\leq c,\ \ \lambda\geq\lambda_0.
\end{align}
We may take $\lambda\geq\lambda_0$ such that
\begin{align}\label{UVV}
\| u\|_\8+\|\nn u\|_\8\leq\ff 1 2,
\end{align}so that $\Theta(t,x):=x+u(t,x)$ satisfies
\begin{align}\label{Theta}
\ff{1}{2}|x-y|^2\leq|\Theta(t,x)-\Theta(t,y)|^2\leq2|x-y|^2,\ \ t\in[0,T],\ x,y\in\R^d.
\end{align}
For any $t\in[0,T]$, let $Y(t)=\Theta(t,X(t))$. By It\^o's formula in \cite[Theorem 1.2.3(3)]{RW24} and \eqref{PDE1}, we obtain
\begin{align*}
\d Y(t)=\big\{\lambda u(t,X(t))+\nn\Theta(t,X(t))b^{(1)}(t,X_t)\big\}\d t+\big(\{\nn\Theta\}\si\big)(t,X(t))\d W(t).
\end{align*}
Combining this with $(A_1)$, $(A_2)$, the second inequality in $(A_3')$, \eqref{UVV} and  It\^o's formula, we find  a constant $c_0>0$ such that
\begin{equation}\label{dy2}
\begin{aligned}
\d \{|Y(t)|^2+1\}^{\alpha}\leq& c_0 \{|Y(t)|^2+1\}^{\alpha-\ff 1 2}\{1+\|Y _t\|^{\alpha}_\tau\}\d t+\d M(t),
\ \ t\in[0,T],
\end{aligned}
\end{equation}
where $$\d M(t)=2\alpha\{|Y(t)|^2+1\}^{(\alpha-1)}\<Y(t),\big(\{\nn\Theta\}\si\big)(t,X(t))\d W(t)\>.$$
Note that
\beq\label{EPR} \begin{split}
&\e^{p\tau t}\|X_t\|_\tau^p:= \sup_{s\in(-\8,0]}\left( \e^{p\tau (t+s)}|X(t+s)|^p\right) \\
& \leq \sup_{s\in(-\8,0]}\left( \e^{p\tau s}|X(s)|^p\right) +\sup_{s\in[0,t]}\left( \e^{p\tau s}|X(s)|^p\right) \\
&= \|X_0\|_\tau^p+\sup_{s\in[0,t]}\left( \e^{p\tau s}|X(s)|^p\right),\ \ t\geq0.
\end{split}\end{equation}
Combining this with \eqref{dy2}, Young's inequality, It\^o's formula, and    $(s+1)^{r}\leq 1$ for   $s\geq 0$ and $r\leq0$, we find a constant $c_1>0$ such that
\begin{equation}\label{dy3}
\begin{split}
\d \big\{\e^{2\alpha \tau t}\{|Y(t)|^2+1\}^{\alpha}\big\}&\leq  \e^{2\alpha \tau t}\d M(t)+c_1\left(1+\|Y_0\|_\tau^{2\alpha}\right)\d t\\
&+ \sup_{s\in[0,t]}  \e^{2\alpha \tau s}\left(|Y(s)|^2+1\right)^{\alpha}\d t,\ \ t\in [0,T].
\end{split}
\end{equation}
So, letting
\beg{align*} &l(t)=\e^{2\alpha \tau t}\{|Y(t)|^2+1\}^{\alpha}+c_1\left(1+\|Y_0\|_\tau^{2\alpha}\right),\\ &\bar{l}_t=\sup_{s\in[0,t]}l(s)\ \ t\in [0,T],\end{align*} we   find a   constant $c>0$ and a martingale $\tt M(t) $ such that
\begin{align}
\d l(t)\leq c\,\bar{l}_t\d t+\d \tt M(t).
\end{align}
Then by \cite[Lemma 3.1]{BWY}, we find a constant $c_2>0$ such that
\begin{align*}
&\E\left[\exp\bigg(\ff{\vv}{T\e^{1+cT}}\int_{0}^{t}\bar{l}_s
\,\d s\bigg)\bigg|\F_0\right] \leq  \e^{\vv l(s)+1}\left(\E\left[\exp\left(2\vv^2\<\tt M\>(t)\right)\big|\F_0\right]\right)^{\ff{1}{2}}\\
&\leq  \e^{\vv l(0)+1}\left(\E\left[\exp\left(c_2\vv^2\int_{0}^{t}l(s)\,\d s\right)\bigg|\F_0\right]\right)^{\ff{1}{2}},\ \ \vv>0,\ t\in [0,T].
\end{align*}
Taking $ \vv= \ff{1}{c_2T\e^{1+cT}},$ we derive
\begin{align*}
\E\left[\exp\bigg(\ff{\vv}{T\e^{1+cT}}\int_{0}^{t}\bar{l}_s
\,\d s\bigg)\bigg|\F_0\right]\leq \e^{2(\vv l(0)+1)},\ \ t\in[0,T].
\end{align*}
By combining   this with \eqref{UVV} and \eqref{EPR}, we derive \eqref{exp0} for  $\beta_T=\ff{\vv}{T\e^{1+cT}} $ and some constant $k_T>0$.
\end{proof}

In the next result,  we present   Krylov-Khasminskii   estimates for SDEs with memory.

\begin{lem}\label{KEs} Assume $(A_1)$, $(A_2)$ and $(A_3')$. For any constants { $T,p,q>1$ such that $(2p,2q)\in\scr K$   and $\vv>0$, there exist  constants  $c=c(T,p,q,\vv)>0$ and $\theta=\theta(T,p,q)>0$ } such that the solution to \eqref{E-2}  satisfies
\begin{equation}\label{KE1}
\begin{aligned}
\E\left[\int_{s}^{t}|f_u(X(u))|\,\d u\bigg|\F_{s}\right]\leq \left(c  +\vv\|X_s\|_\tau^\alpha\right)\|f\|_{\tt L_q^p(s,t)}^2,
\end{aligned}
\end{equation}
\begin{equation}
\begin{aligned}\label{KE2}			\E\left[\exp\left(\int_{s}^{t}|f_u(X(u))|\,\d u\right)\bigg|\F_{s}\right]\leq \exp\left[\vv{\|X_{s}\|}_\tau^{2\alpha}+c \left( 1+\|f\|^{\theta}_{\tilde{L}^p_q(s,t)}\right) \right]
\end{aligned}
\end{equation}	  for any $0\leq s\leq t\leq T$ and $f\in\tilde{L}^p_q(T)$.
\end{lem}
\begin{proof} Without loss of generality, we may simply prove these estimates for $s=0$.

For any $x\in\R^d$, let $\phi^x\in\C_\tau$ be defined as
$$\phi^x(r):=x,\ \ r\le 0.$$ For any $t\in [0,T],\ x\in\R^d $ and $\xi\in\C_\tau$, let
$$\bar b(t,x):= b^{(0)}(t,x)+ b^{(1)}(t, \phi^x),\ \ \ \hat b(t,\xi):= b^{(1)}(t,\xi)-b^{(1)}(t,\xi^0).$$
By \eqref{UVV}, \eqref{EPR},   \eqref{dy3} and  the Burkholder-Davis-Gundy (BDG) inequality, we find a constant $k_1>0$  such that
\begin{align}\label{EX2A}
\E\left[\sup_{t\in[0,T]}\|X_t\|_\tau^{2\alpha}\bigg|\F_0\right]\leq k_1\left(1+\|X_0\|_\tau^{2\alpha}\right).
\end{align}
For any $n\ge 1$, let
\beq\label{TNN} \tau_n:= \inf\{t\in[0,T]:\|X_t\|_\tau\geq n\},\ \ \inf\emptyset:=T.\end{equation}
By \eqref{EX2A}, $(A_1)$, $(A_2)$, the second inequality in $(A_3')$, and applying Krylov's estimate in \cite[Theorem 3.1]{ZY} for $\bar b(t,x) $ in place of $b(t,x)$, see also \cite[Theorem 1.2.3(2)]{RW24},
we find  constants $k_2$ depending on $\vv$ and $k_3>0$ independent of $\vv$ such that
\begin{equation}
\begin{aligned}
&\E\left[\int_{0}^{t\wedge\tau_n}|f_s|\left(X(s)\right) \,\d s\bigg|\F_0\right]\\
\leq& \left\{k_2+\vv\left[\E\left(\int_{0}^{t\wedge\tau_n}|\hat b (s,X_s)|^2\,\d s \bigg|\F_0\right)\right]^{\ff 1 2}\right\}\|f\|_{\tt L_q^p(0,t)}\\
\leq& k_3\left(k_2+\vv\|X_0\|_\tau^\alpha\right)\|f\|_{\tt L_q^p(0,t)}, \ \  t\in [0, T].
\end{aligned}
\end{equation} Letting $n\to\8$, we derive \eqref{KE1} since $\vv>0$ is arbitrary.

To  prove \eqref{KE2}, we shall use Girsanov's transform. Let
\begin{align}\label{Gam}
\gamma_t=\{\si^*(\si\si^*)^{-1}\}(t,X(t)) \hat b(t, X_t),\ \ t\in[0,T].
\end{align}
Then   for any $n\ge 1$ and $\tau_n$ in \eqref{TNN},
$$
R_n(t):=\exp\bigg[-\int_{0}^{t\land \tau_n}\<\gamma_r,\d W(r)\>-\ff{1}{2}\int_{0}^{t\land\tau_n}|\gamma_r|^2\,\d r\bigg],\ \ t\in [0,T]
$$ is a   martingale. By
Girsanov's theorem,
$${W}_n(t):=W(t)+\int_{0}^{t\wedge\tau_n}\gamma_s\,\d s,\ \ t\in [0,T]$$
is a Brownian motion under the probability measure
$\d \Q_n:=R_{n}(T) \d\P$,
and $X$ solves  the equation
\begin{align}\label{XN}X(t)=X(0)+\int_{0}^{t} \bar b(s,X(s)) \,\d s +\int_{0}^{t}\si(s,X(s))\,\d   W_n(s), \ \ t\in[0,\tau_n].\end{align}
Since the Krylov estimate in \cite[Theorem 1.2.4]{RW24} holds for \eqref{XN}, letting $f^n_r(x)=f_r(x)\1_{\{|x|<n\}}$, by the argument in the proof of \cite[Lemma 4.4]{Z24}, we find a constant
$k>0$ independent of $s$ such that for any $t\in[0,T]$, 
\begin{equation}\label{*}
\E_{\Q_n}\left[\exp\left(\lambda\int_{0}^{t} \left|f^n_{r}\left(X(r\wedge\tau_n)\right)\right| \,\d r\right)\bigg|\F_0\right]
\leq \exp\Big[k+ k \|f\|_{\tt L_q^p(0,t)}^{k}\Big].
\end{equation}
On the other hand, for constants $K$ in $(A_1)$ $(A_2)$,  $(A_3')$, and $\beta_T$   in Lemma \ref{EXPB}, let
$$ p_0:=\ff{3+\sqrt{1+8\beta_TK^{-3}}}{4}.$$ We have
$$  k(p):=(p-1)(2p-1)K^3\in (0, \beta_T],  \ \ \ p\in (1,p_0].$$
Thus,   taking $\tt p\in[1,p_0]$ with $\tt p-1$ small enough such that ${k_T} k(\tt p) \leq 2\vv$, by  $(A_1)$ $(A_2)$,  $(A_3')$ and \eqref{EXPB1},  we  find a constant $\bb=\bb(p,q,\vv)> 0$ such that
\begin{equation}\label{PK3}
\begin{aligned}
&\E\big[|R_n(T)|^{-(\tt p-1)}\big|\F_{0}\big]\\=&\E\left[\exp\bigg((\tt p-1)\int_{0}^{T\wedge\tau_n}\<\gamma_s,\d W(s)\>+\ff{\tt p-1}{2}\int_{0}^{T\wedge\tau_n}|\gamma_s|^2\,\d s\bigg)\bigg|\F_{0}\right]\\
\leq&\left(\E\left[\exp\bigg((\tt p-1)(2\tt p-1)\int_{0}^{T\wedge\tau_n}|\gamma_s|^2\,\d s\bigg)\bigg|\F_0\right]\right)^{\ff 1 2}\\
&\quad\times \left(\E\left[\exp\bigg(2(\tt p-1)\int_{0}^{T\wedge\tau_n}\<\gamma_s,\d W(s)\> - 2(\tt p-1)^2
\int_{0}^{T\wedge\tau_n}|\gamma_s|^2\,\d s\bigg)\bigg|\F_0\right]\right)^{\ff 1 2}\\
\leq&\exp\left[\bb +\vv\|X_0\|^{2\alpha}_\tau\right]\times 1= \exp\left[\bb+\vv\|X_0\|^{2\alpha}_\tau\right].
\end{aligned}
\end{equation}	
Combining this with \eqref{*} and applying  H\"older's inequality,    we find constants $c=c(p,q,\vv)>0$
and $\theta=\theta(p,q)>0$ such that
\begin{align*}
&\E\left[\exp\left(\int_{0}^{t}\left|f^n_{s}(X(s\wedge\tau_n))\right| \,\d s\right)\bigg|\F_{0}\right]\\	=&\E_{\Q_n}\left[R_n(T)^{-1}\exp\left(\int_{0}^{t}|f^n_{s}(X(s\wedge\tau_n))|\, \d s\right)\bigg|\F_{0}\right]\\	
\leq&\left(\E\big[R_n(T)^{-(\tt p-1)}\big|\F_{0}\big]\right)^{\ff{1}{\tt p}}\left(\E_{\Q_n}\left[\exp\left(\ff{\tt p}{\tt p -1}\int_{0}^{t}|f^n_{r}  (X(s\wedge\tau_n))|\, \d s\right)\bigg|\F_{0}\right] \right)^{\ff{\tt p -1}{\tt p}} \\
\leq& \exp\left[\vv\|X_0\|_\tau^{2\alpha}+c \left(1+\|f\|^{\theta}_{\tilde{L}_{p}^{q}(0,t)}\right)\right].
\end{align*}
This with  $n\to\8$ implies \eqref{KE2}.
\end{proof}

By Lemma \ref{EXPB} and \eqref{KE2}, we have the following corollary which is a key  in the  proofs of Theorem \ref{TB1} and \ref{TN}.

\begin{cor}\label{cor1} Assume $(A_1)$, $(A_2)$ and $(A_3')$.
\beg{enumerate}
\item[$(1)$] For any $T\in (1,\infty)$, $(p,q)\in\scr K$ and  $\vv\in(0,1)$, there exist constants $k=k(T,p,q,\vv)>0$ and $\theta=\theta(T,p,q)>0$  such that    any solution $X_s$ to $\eqref{E-2}$ satisfies
\begin{equation}
\begin{aligned}\label{KE3}	&	\E\left[\exp\left(\int_{s}^{t}|f_u(X(u))|  \|X_u\|_\tau^{\alpha}\,\d u\right)\bigg|\F_{s}\right]
\leq \exp\left[ \vv{\|X_{s}\|}_\tau^{2\alpha}+k\left( 1+\|f\|^{\theta}_{\tilde{L}^p_q(s,t)}\right) \right],\\
&\qquad \ 0\leq s\leq t\leq T,\ f\in\tilde{L}^p_q(T).
\end{aligned}
\end{equation}
\item[$(2)$] 	For any $T\in (1,\infty)$, there exist  constants $\beta_T, k_T>0$
such that any solution to $\eqref{E-2}$ satisfies
\begin{equation}
\begin{aligned}\label{EXP2}
\sup_{t\in [0,T]}\E\Big[
\e^{\beta\|X_t\|_\tau^{2\alpha}}\Big|\F_0\Big]\leq \e^{ k_T\bb  (1+\|X_0\|_\tau^{2\alpha})},\ \ \bb\in [0,\bb_T].
\end{aligned}
\end{equation}\end{enumerate}
\end{cor}
\begin{proof}
(a) 	  {Let $\bb_T, k_T>0$ be in Lemma \ref{EXPB}, and let $\vv, c>0$ be in Lemma \ref{KEs}}. We choose  $\beta \in (0,\bb_T]$   such that $k_T \beta \leq 2\vv$, so that by Lemma \ref{EXPB}, \eqref{KE2} and Young's and H\"older's  inequalities, we find constants $\theta=\theta(T,p,q), k=k(T,p,q,\vv)>0$ such that
\begin{equation*}
\begin{aligned}	&\E\left[\exp\left(\int_{s}^{t}|f_u(X(u))|\cdot  \|X_u\|_\tau^{\alpha}\,\d u\right)\bigg|\F_{s}\right]\\\leq&\E\left[\exp\left(\int_{s}^{t}\left({\ff 1 {2\bb}}|f_u(X(u))|^2 +\ff{\beta}{2}\cdot\|X_u\|_\tau^{2\alpha}\right)\,\d u\right)\bigg|\F_{s}\right]\\
\leq&\left( \E\left[\exp\left({\ff 1 \bb}\int_{s}^{t}|f_u(X(u))|^2\,\d u\right)\bigg|\F_{s}\right]\right) ^{\ff 1 2}\left( \E\left[\exp\left(\beta\int_{s}^{t}\|X_u\|_\tau^{2\alpha}\,\d u\right)\bigg|\F_{s}\right]\right) ^{\ff 1 2}\\
\leq& \exp\left[ \vv{\|X_{s}\|}_\tau^{2\alpha}+k\left( 1+{
\|f^2\|^{\theta/2}_{\tilde{L}^{p/2}_{q/2}(s,t)}}\right) \right]= \exp\left[ \vv{\|X_{s}\|}_\tau^{2\alpha}+k\left( 1+  \|f\|^\theta_{\tilde{L}^p_q(s,t)} \right) \right].
\end{aligned}
\end{equation*}	

(b) In the case where $b^{(1)}=0$, \eqref{EXP2} can be obtained by repeating the proof of \cite[Lemma 2.4]{BS} and \cite[Corollary 2.5]{B18}, so that by  Girsanov's theorem,  for { the segment solution to \eqref{XN} and $\Q_n$} in the proof of Lemma \ref{KEs},
there exists constants $c_0,c_1>0$ such that
\beq\label{o8}
\E_{\Q_n}
\left[\e^{c_0\|X_{t\land \tau_n}\|_\tau^{2\alpha}}\big|\F_0\right]\leq \e^{c_1+ c_1\|X_0\|_\tau^{2\alpha}},\ \ \ t\in[0,T],\ n\ge 1.
\end{equation}
Letting $\tt p>1$ be in  \eqref{PK3} and taking
$$\bb_T:=\ff {\tt p c_0}{\tt p-1},$$
by combining   \eqref{PK3} with  \eqref{o8} and using H\"older's inequality, we find a constant $c>0$ such that
\begin{equation}
\begin{aligned}& \E\big[\e^{\beta_T \|X_{t\land \tau_n}\|_\tau^{2\alpha}}\big|\F_0\big]\leq \E_{\Q_n}\left[R_n(T)^{-1}\e^{\bb_T \|X_{t{\land \tau_n}}\|_\tau^{2\alpha}}\big|\F_0\right]\\ \leq&\left(\E\big[R_n(T)^{-(\tt p-1)}\bigg|\F_0\big]\right)^{\ff{1}{\tt p}}\left(\E_{\Q_n}\left[\exp\left(\ff{\tt p\beta_T}{\tt p -1}\|X_{t{\land \tau_n}}\|_\tau^{2\alpha}\right)\bigg|\F_0\right] \right)^{\ff{\tt p -1}{\tt p}} \\
\leq& \exp\left(c+c\|{X_0}\|_\tau^{2\alpha}\right), \ \ t\in[0,T],\ n\ge 1.
\end{aligned}
\end{equation} By letting $n\to\infty$ we derive 
$$\E\big[\e^{\beta_T \|X_{t}\|_\tau^{2\alpha}}\big|\F_0\big]\le \exp\left(c+c\|{X_0}\|_\tau^{2\alpha}\right), \ \ t\in[0,T],$$
which implies \eqref{EXP2} for ${k}_T:=\ff{c}{\bb_T}$
by Jensen's inequality.
\end{proof}

\begin{proof}[Proof of Theorem \ref{TB1}] Let $X_0$ be an $\F_0$-measurable random variable in $\C_\tau$.
Take $\psi\in C_b^{\infty}([0,\infty))$ such that $0\leq \psi\le1$, $\psi(u)=1$ for $u\in[0,1]$, and $\psi(u)=0$ for $u\in[2,\infty)$. Define
\begin{align}\label{appro}
b^{(1)}_n(t,\xi)=b^{(1)}(t,\xi)\psi\left( \ff{\|\xi\|_\tau}{n}\right),\ \ t\geq0,\ \xi\in\C_\tau.
\end{align} By \eqref{BL2},
for any $n\geq 1$, $b^{(1)}_n$ is bounded and satisfies $(A_3')$. So, Lemma \ref{LB1} ensures that
the SDE
$$
\d X^n(t)=b^{(1)}_n(t,X^n_t)\d t+b^{(0)}(t,X^n(t))\d t+\si(t,X^n(t))\d W(t), \ \ X_0^n=X_0,\  t\in[0,T]$$
has a unique   strong solution.
Let $\tau_n$ be in \eqref{TNN}.
It suffices to prove that   for any $k\ge 2$ there exists a constant $c(T,k)>0$ such that
\beq\label{C*} \E\bigg[\sup_{t\in [0,T]}\|X_{t}^n\|_\tau^k\bigg|\F_0\bigg]\le c(T,k)\big(1+ \|X_0\|_\tau^k\big),\ \ n\ge 1.\end{equation}
This implies $\P(\tau_n=T)\to 1$ as $n\to\infty$ and
$$X(t):= 1_{\{t=0\}}X(0)+\sum_{n=1}^\infty X_n(t)1_{(\tau_{n-1},\tau_n]}(t),\ \ t\in [0,T]$$
is the unique solution to \eqref{E-2} and \eqref{EV0} holds, where $\tau_0:=0$. To prove \eqref{C*}, we use Zvonkin's transform.

For  $\lambda\ge \ll_0$ and $u$ solving the   PDE \eqref{PDE1} such that \eqref{LAMU} and \eqref{UVV} hold, let
$$\Theta(t,x):=x+u(t,x),\ \   Y^n(t):=\Theta(t,X^n(t)), \ \ t\in[0,T].$$ By It\^o's formula, we obtain
\begin{equation}\begin{aligned}\label{1DY}
\d Y^n(t)=&\{\nn\Theta(t,X^n(t))b^{(1)}_n(t,X_t^n)+\lambda u(t,X^n(t))\}\d t\\
&+\{(\nn\Theta\si)(t,X^n(t))\}\d W(t).
\end{aligned}\end{equation}
Let
$$Z(t)=\e^{k\tau t}|Y^n(t)|^k+1.$$ By $(A_1)$--$(A_3)$, \eqref{UVV} and It\^o's formula, we find a constant
$c_1>0$ such that
\begin{align}\label{zj1}
\d Z(t)\leq
&c_1\left(1+\|Y^n_0\|_\tau^{k}\right)\d t+c_1\sup_{s\in[0,t]}Z(s) \d t+\d M(t),\ \ t\leq\tau_n,
\end{align}  where
$$\d M(t)=k\e^{k \tau t}|Y^n(t)|^{j-2}\<Y^n(t),\big(\{\nn\Theta\}\si\big)(t,X^n(t))\d W(t)\>.$$
By  the BDG inequality and the Young inequality, we find  constants $c_2,c_3>0$ such that
\begin{align*}
&\E\bigg[\sup_{s\in[0,t\land \tau_n]}M(s)\bigg|\F_0\bigg]\leq c_2\E\bigg[\int_{0}^{t}Z(s\land\tau_n)^2\,\d s\bigg|\F_0\bigg]^{\ff{1}{2}}\notag\\ \leq&\ff{1}{2}\E\bigg[\sup_{s\in[0,t\land \tau_n]}Z(s)\bigg|\F_0\bigg]+c_3\int_{0}^{t}\E\big[Z(s\land\tau_n)\big|\F_0\big]\,\d s,\ \ t\in[0,T].
\end{align*}
Combining this with \eqref{EPR}, \eqref{zj1} and
$$|X^n(t)-Y^n(t)|\leq\|u\|_\8<\8,$$ we may  apply  Gronwall's inequality to derive the desired \eqref{C*} for some constant $c=c(T,k)>0.$ 	 \end{proof}

\begin{proof}[Proof of Theorem \ref{TN}]  Let $\Theta (t,x):=x+u(t,x)$ for $u$ solving the   PDE \eqref{PDE1} such that \eqref{LAMU} and \eqref{UVV} holds, and denote
\beg{align*}
&Y^\xi(t):=\Theta(t,X^\xi(t)),\ \ \  Y^\eta(t):=\Theta(t,X^\eta(t)),\\
&\Phi(t):=\e^{\tau t}(X^\xi(t)-X^\eta(t)),\ \ \Xi(t):=\e^{2k\tau t}|Y^\xi(t)-Y^\eta(t)|^{2k},\\
&	g_1(t):= \M\|\nn^2u\|(t,X^\xi(t))+ \M\|\nn^2u\|(t,X^\eta(t)),\\
& g_2(t):=g_1(t)+ \M\|\nn\si\|(t,X^\xi(t))+ \M\|\nn\si\|(t,X^\eta(t)),\\
&A(t):= \int_{0}^{t}\left(1+g_2(s)^2 +g_1(s)\|X^\xi_s\|_\tau^{\alpha}\right)\d s,\ \ \  t\in [0,T].
\end{align*}
By Lemma \ref{Hardy}, Lemma \ref{KEs}, Corollary \ref{cor1} and   Holder's inequality,
for any $\bb>0$ and $\vv\in (0,1)$ we find a constants $c(\bb,\vv)>0$  such that
\begin{align}\label{Edelta}
\E\left[\e^{\bb {A}(T)}\right]\leq \e^{c(\bb,\vv)+\vv(\|\xi\|_\tau^{2\alpha}+\|\eta\|_\tau^{2\alpha})}.
\end{align}
By $(A_1)$, $(A_2)$, $(A_3')$, \eqref{UVV} and  Lemma \ref{Hardy}, we find a constant $k_0>0$ such that
\begin{align*}
&\left|[\nn\Theta(t,X^\xi(t))]b^{(1)}(t,X^\xi_t)-[\nn\Theta(t,X^\eta(t))]b^{(1)}(t,X^\eta_t)\right|\\
&\leq k_0\|X^\xi_t-X_t^\eta\|_\tau+ k_0\left(1+g_1(t)\right)\|X_t^\xi\|_\tau^\alpha|X^\xi(t)-X^\eta(t)|,\end{align*}
$$\left\|\left([\nn\Theta]\si\right)(t,X^\xi(t))-\left( [\nn\Theta]\si\right)(t,X^\eta(t))\right\|^2_{HS}
\leq k_0\left(1+g_2 (t)^2\right)|X^\xi(t)-X^\eta(t)|^2.$$
Combining this with \eqref{Theta}, \eqref{EPR},  \eqref{1DY} and It\^o's formula, for any $k\geq1$, we find a martingale $M(t)$ and a constant $c_1(T,k)>0$ such that for any $t\in[0,T]$,
\begin{align*}
\d\Xi(t)\leq {c_1(T,k)}\left\{\Xi(t)\d {A}(t)+\sup_{s\in[0,t]}\Xi(s)\d t +\|Y_0^\xi-Y_0^\eta\|^{2k}_\tau\d t\right\}+\d M(t).
\end{align*}
Using the It\^o formula, the stochastic Gr\"onwall inequality in \cite[Lemma A.5]{B18} and \eqref{Theta},  there exists a constant $c_2(T,k)>0$ such that
$$
\left(\E\left[\sup_{s\in[0,t]}\left(\e^{-{c_1(T,k)}{A}(s)}|\Phi(s)|^{2k}\right)^{\ff 2 3}\right]\right)^{\ff 3 2}\leq c_2(T,k)\|\xi-\eta\|_\tau^{2k}, \ \ t\in[0,T].
$$
By combining this   with \eqref{Edelta} and   H\"older's inequality, we find a constant   $c_3(\vv,T,k)>0$, such that
\begin{equation}
\begin{aligned}
&\E\left[\sup_{s\in[0,t]}|\Phi(s)|^k\right] \leq\E\left[\e^{\ff 1 2{c_1(T,k)} A(t)}\left(\sup_{s\in[0,t]}
\e^{-\ff 1 2 {c_1(T,k)}A(s)}|\Phi(s)|^{k}\right)\right]\\
&\leq \left(\E\left[\e^{2{c_1(T,k)}A (t)}\right]\right)^{\ff 1 4}\left( \E\left[\sup_{s\in[0,t]}\left(\e^{-{c_1(T,k)}{A}(s)}|\Phi(s)|^{2k}\right)^{\ff 2 3}\right]\right)^{\ff{3}{4}}\\
&\leq c_3(\vv,T,k)\e^{\vv
(\|\xi\|_\tau^{2\alpha}+\|\eta\|_\tau^{2\alpha})}\|\xi-\eta\|^k_\tau, \ \ t\in[0,T].
\end{aligned}
\end{equation}
This  together with \eqref{EPR}  yields the desired estimate \eqref{E01} for some constant $c >0$ depending on $\vv,T$ and $k$.
\end{proof}

\section{Proof of Theorem \ref{T01}} It suffices to prove the assertions up to any fixed time  $T\in (0,\infty).$   To this end, we   use the fixed point argument as in \cite{HW2}.

Let   $X_0$ be an $\F_0$-measurable random variable on $\C_\tau$ with $\gg:=\L_{X_0}\in \scr P_{k,\e}^\alpha(\C_\tau).$
For any constant   $N\geq 2$, let
\begin{align*}
\mathcal C^{\gamma,N}_k:=\left\{\mu\in C_b^w([0,T];\scr P_{k,\e}^\alpha{(\C_\tau)}):\mu_0=\gamma, \sup_{t\in[0,T]}\e^{-Nt}(1+\mu_t(\|\cdot\|_\tau^k))\leq N\right\},
\end{align*}
{where	\begin{equation*}\label{Cw}
\begin{split}
&C^{w}([0,T];\scr P_{k,\e}^\alpha(\C_\tau)):=\{\mu:[0,T]\to\scr P_{k,\e}^\alpha(\C_\tau)) \text{ is weakly continuous}\},\\
&C^{w}_b([0,T];\scr P_{k,\e}^\alpha(\C_\tau)):=\Big\{\mu\in C^{w}([0,T];\scr P_{k,\e}^\alpha(\C_\tau)):\sup_{t\in[0,T]}\W_{k,var}(\mu_t,\mu_0)<\8\Big\}.
\end{split}
\end{equation*}}
Then as $N\uparrow\8$,
\begin{align}\label{cn}
\mathcal C^{\gamma,N}_k\uparrow\mathcal C^\gamma_k:=\left\{\mu\in C_b^w([0,T];\scr P_{k,\e}^\alpha{(\C_\tau)}):\mu_0=\gamma\right\}.
\end{align}
According to Theorem \ref{TB1} and Corollary \ref{cor1}, under the assumptions $(H_1)$--$(H_2)$, for any $\mu\in\mathcal C_k^\gamma$, the SDE
\begin{align}\label{E-3}
\d X^\mu(t)=b(t,X_t^\mu,\mu_t)\d t+\si(t,X^\mu(t))\d W(t), \ \ t\in[0,T],  X_0^\mu =X_0
\end{align}
has a unique segment  solution with
$$\Phi_\cdot^\gamma\mu:=\L_{X_\cdot^\mu}\in\mathcal C_k^\gg.$$
Then the well-posedness of \eqref{E-1} follows if the map $\Phi^{\gamma}$ has a unique fixed point in $\mathcal C^\gamma_k$. To this end, we need to verify that {there exists a constant $N_0\ge 2$ such that for any $N\ge N_0,$ the following two assertions hold:}
\beg{enumerate}
\item[(a)] 
$\Phi^\gg:\mathcal C_k^{\gg,N}\to\mathcal C_k^{\gg,N},$ i.e.
\begin{align}\label{p1}
\sup_{t\in[0,T]}\e^{-Nt}\left(1+\|\Phi^\gamma_t\mu\|_k^k\right)\leq N,\ \ \mu\in\mathcal C^{\gamma,N}_k.
\end{align}
\item[(b)]$\Phi^\gg$ has a unique fixed point in $\mathcal C_k^{\gg,N}.$
\end{enumerate}
Once these two assertions are confirmed, $\Phi^\gg$ has a unique fixed point $\mu\in \mathcal C_k^\gg$, and
$X_t=X_t^\mu$ is the unique segment solution of  \eqref{E-1} with initial value $X_0$,
so that  \eqref{xn} follows from \eqref{xn2}.

In the following two subsections, we prove assertions (a) and (b) respectively.

\subsection{Proof of (a)} It suffices to   prove  \eqref{p1} for $k>0$ and large $N$.
By {Lemma \ref{KEs}},   $(H_1)$-$(H_2)$ implies that for some constant $c_1>0$  we have
\begin{align}\label{ke3}
\E\left[\exp \left(\int_0^T|f_0(s,X^\mu(s))|^2\,\d s\right)  \right]\leq c_1\E[\e^{\vv\|X_0^\mu\|_\tau^{2\alpha}}]=c_1\gg(\e^{\vv\|\cdot\|_\tau^{2\aa}})<\infty.
\end{align}
Hence,
$$\sup_{\mu\in \mathcal C_k^\gg}\E\bigg(\int_0^T |f_0(s,X^\mu(s))|^2\,\d s\bigg)^k<\infty.$$
Combining this with  $(H_1)$-$(H_2)$, $\gamma:=\L_{X_0}\in\scr P_{k,\e}^\alpha(\C_\tau)$,  the It\^o formula and the BDG inequality, we find constants $c_2,c_3>0$ such that
\begin{equation}\label{xk1}
\begin{aligned}
&\E\left[\sup_{s\in[0,t]}\left(1+\e^{k\tau s}|X^\mu(s)|^k\right)\right]\leq c_2\E (1+|X(0)|^k)\\&+c_2\E \left(\int_0^t
\e^{\tau s}\left\{\|X^\mu_s\|_\tau+|f_0(s,X^\mu(s))|+\|\mu_s\|_k\right\}\,\d s\right)^k\\
&\leq c_3+c_3\E \left(\int_0^t\left\{\e^{2\tau s}\|X_s^\mu\|_\tau^2+\e^{2\tau s}\|\mu_s\|_k^2\right\}\d s\right)^{\ff k 2},\ \ t\in[0,T].
\end{aligned}
\end{equation}
Below we verify \eqref{p1} by considering $k\ge 2$ and $k\in (0,2)$ respectively.

\paragraph{$(a_1)$} Let $k\geq2$. By \eqref{EPR} and \eqref{xk1}, there exists a constant $c_4>0$ such that for any $t\in[0,T]$,
\begin{equation*}
\begin{aligned}
&\E\left[ \sup_{s\in[0,t]}\left(1+\e^{k\tau s}\|X_s^\mu\|_\tau^k\right)\right] \leq c_4+c_4\int_0^t\left\{\E\left[\e^{k\tau s}\|X_s^\mu\|_\tau^k\right]+\e^{k\tau s}\|\mu_s\|_k^k\right\}\,\d s.
\end{aligned}
\end{equation*}
Using Gr\"onwall's lemma, it holds
\begin{align*}
\E\left[ \sup_{s\in[0,t]}\left(1+\e^{k\tau s}\|X_s^\mu\|_\tau^k\right)\right] \leq c_4\e^{c_4 T}\left(1+\int_0^t \e^{k\tau s }\|\mu_s\|_k^k\,\d s\right),\ \ t\in[0,T].
\end{align*}
Hence, by noting that  $\mu\in\mathcal C^{\gamma,N}_k$, we can find a constant $c_5$ such that
\begin{equation}\label{muk1}
\begin{split}
&\E\left[1+\|X_t^\mu\|_\tau^k\right]\leq1+\e^{-k\tau t}	\E\left[ \sup_{s\in[0,t]}\left(1+\e^{k\tau s}\|X_s^\mu\|_\tau^k\right)\right]\\ &\leq c_5+c_5\int_0^t\e^{-k\tau (t-s)}\|\mu_s\|_k^k\,\d s\\
&\leq c_5+c_5N\e^{Nt}\int_0^t\e^{-N(t-s)}\d s\leq 2c_5\e^{Nt},\ \ t\in[0,T].
\end{split}
\end{equation}
Taking $N_0=2c_5$, we derive
\begin{equation*}
\begin{split}
&\sup_{t\in[0,T]}\e^{-Nt}\left(1+\|\Phi_t^\gamma\mu\|_k^k\right)=	\sup_{t\in[0,T]}\e^{-Nt}\E\left(1+\|X_t^\mu\|_\tau^k\right){\le N}, \quad  N\geq N_0,\ \mu\in\mathcal C^{\gamma,N}_k.
\end{split}
\end{equation*}

\paragraph{$(a_2)$} Let $k\in(0,2)$. By \eqref{EPR} and \eqref{xk1}, we find constants $c_6,c_7,c_8>0$ such that
\begin{equation*}\label{xk3}
\begin{split}
&U_t:=\E\left[\sup_{s\in[0,t]}\left(1+\e^{k\tau s}\|X_s^\mu\|_\tau^k\right)\right]\leq c_6+c_6\E \left(\int_0^t\left\{\e^{2\tau s}\|X_s^\mu\|_\tau^2+\e^{2\tau s}\|\mu_s\|_k^2\right\}\,\d s\right)^{\ff k 2} ,\\
&\leq c_7+c_7\left(\int_0^t\e^{2\tau s}\|\mu_s\|_k^2\,\d s\right)^{\ff k 2}
+c_7\E\left\{\left[\sup_{s\in[0,t]}\e^{k\tau s}\|X_s^\mu\|_\tau^k\right]^{1-\ff k 2}\left(\int_0^t\e^{k\tau s}\|X_s^\mu\|_\tau^k\,\d s\right)^{\ff k 2}\right\}\\
&\leq c_7+c_7\left(\int_0^t\e^{2\tau s}\|\mu_s\|_k^2\,\d s\right)^{\ff k 2}+\ff 1 2U_t+c_8\int_0^t U_s\,\d s,\ \ t\in[0,T],
\end{split}
\end{equation*}
then Gr\"onwall's lemma implies that 
\begin{align*}
U_t\leq 2c_7\e^{2 c_8T}\left\{1+\left(\int_0^t \e^{2\tau s}\|\mu_s\|_k^2\,\d s\right)^{\ff k 2}\right\},\ \ t\in[0,T].
\end{align*}
Thus, there exist constants $c_9,c_{10}>0$ such that for any $\mu\in\mathcal C^{\gamma,N}_k$,
\begin{equation}\label{muk2}
\begin{split}
&\E\left(1+\|X_t^\mu\|_\tau^k\right)\leq 1+\e^{-k \tau t}U_t\leq c_9+c_9\left(\int_0^t\e^{-2\tau(t-s)}\|\mu_s\|_k^2\,\d s\right)^{\ff k 2}\\
&\leq c_9+c_9Ne^{Nt}\left(\int_0^t\e^{-2N(t-s)}\,\d s\right)^{\ff k 2}\leq c_9+c_{10}N^{1-\ff k 2}\e^{Nt},\ \ t\in[0,T].
\end{split}
\end{equation}
Thus, there exists a constant $N_0>2$ such that for any $N\geq N_0$,
\begin{equation*}
\begin{split}
&\sup_{t\in[0,T]}\e^{-Nt}\left(1+\|\Phi_t^\gamma\mu\|_k^k\right)=	\sup_{t\in[0,T]}\e^{-Nt}\E\left(1+\|X_t^\mu\|_\tau^k\right)\\
&\leq c_9+c_{10}N^{1-\ff k 2}\leq N,\ \ \mu\in\mathcal C^{\gamma,N}_k.
\end{split}
\end{equation*}

\subsection{Proof of (b)} To ensure that $\Phi^\gg$ has a unique fixed point in $\mathcal C_k^{\gg,N}$ for $N\ge N_0$, we shall prove that it is contractive under a complete metric.

For any $\t>0$, let
\begin{align*}
&\W_{k,\t,var}(\mu,\nu):=\sup_{t\in[0,T]}\e^{-\t t}\|\mu_t-\nu_t\|_{k,var},\\
&\W_{k,\t}(\mu,\nu):=\sup_{t\in[0,T]}\e^{-\t t}\W_k(\mu_t,\nu_t),
\ \ \mu,\nu\in\mathcal C^{\gamma,N}_k.
\end{align*}
Then   the metric
$\tt\W_{k,\t,var}=\W_{k,\t,var}+\W_{k,\t}$ is complete on $\mathcal C_k^{\gg,N}$.
We intend to show that $\Phi^\gamma$ is contractive in $C^{\gamma,N}_k$ under the metric $\tt\W_{k,\t,var}$ when $\t$ is large enough.

By Theorem \ref{TB1}, $(H_1)$-$(H_2)$,   \eqref{p1} and $\gg:=\L_{X_0}\in \scr P_k(\C_\tau)$, we find a constant $C_0(N)>0$ such that
\begin{align}\label{xn2}
\sup_{\mu\in C_k^{\gg,N}} \E\Big(\E\Big[\sup_{t\in[0,T]} \|X^\mu_t\|_\tau^{2k}\Big|\F_0 \Big]\Big)^{\ff 1 2}\leq C_0(N)\E \left(1+\|X_0\|_\tau^k\right)<\infty.
\end{align}
Next, since $\mu,\nu\in\mathcal C^{\gamma,N}_k$, we find a constant $C_0(N)>0$ such that 
\begin{align*}
\sup_{t\in[0,T]}\W_{k,var}(\mu_t,\nu_t)\leq C_0(N).
\end{align*}
Thus, by $(H_1)$ and $(H_2)$, we find a constant $C_1(N)>0$ such that for any $\mu,\nu\in\mathcal C^{\gamma,N}_k$,
$$\zeta_s:=\left\{\si^*(\si\si^*)^{-1}\right\}(s,X^\mu(s))\left[b^{(1)}(s,X_s^\mu,\mu_s)-b^{(1)}(s,X_s^\mu,\nu_s)\right],\ \ s\in[0,T]$$
satisfies \begin{align}\label{ztn}
|\zeta_s|^2\leq C_1(N) H(s)\left(1\wedge\W_{k,var}(\mu_s,\nu_s)^2\right),\ \ s\in[0,T].
\end{align}
Recalling that $H\in L_{loc}^1([0,\8);(0,\8))$,   we have \begin{align*}
\int_0^tH(s)\left(1\wedge\W_{k,var}(\mu_s,\nu_s)^2\right)\,\d s\leq \int_0^tH(s)\,\d s<\8,
\end{align*}
then by $(H_2)$ and Girsanov's  theorem, for any $t\in[0,T]$,
\begin{align}\label{rt1}
R_t:=\exp\left(\int_0^t\<\zeta_s,\d W(s)\>-\ff 1 2\int_0^t|\zeta_s|^2\,\d s \right)
\end{align}
is a martingale and
$$\tt W_u:=W_u-\int_0^u\zeta_s\,\d s, \ \ u\in[0,t]$$
is a Brownian motion under the probability measure $\Q_t:=R(t)\P$.
Since $\e^s-1\leq s\e^s$ for $s\geq0$, we find   constants $C_2(N),C_3(N)>0$ such that
\beq\beg{split}\label{N1}  &\E\left(|R_t-1|^2\big|\F_0\right)
=\E\left[\e^{2\int_0^t\<\zeta_s,\d W(s)\>-2\int_0^t|\zeta_s|^2\,\d s+\int_0^t|\zeta_s|^2\,\d s}\bigg|\F_0\right]-1\\
&\le \e^{C_1(N)\int_0^t H(s)\left(1\wedge\W_{k,var}(\mu_s,\nu_s)^2\right)\,\d s}-1
\le C_2(N)\int_0^t H(s)\W_{k,var}(\mu_s,\nu_s)^2\,\d s\\
&\leq C_3(N)\e^{2\t t}\tt\W_{k,\t,var}(\mu,\nu)^2 \int_0^tH(s)\e^{-2\t(t-s) }\,\d s ,\ \ \mu,\nu\in\mathcal C^{\gamma,N}_k.\end{split}\end{equation}
Reformulating \eqref{E-3} as
\begin{align}\label{E-4}
\d X^\mu(r)=b(r,X_r^\mu,\nu_r)\d r+\si(r,X^\mu(r))\d\tt W(r), \ \ \L_{X_0^\mu}=\gamma, r\in[0,t].
\end{align}
By the uniqueness, we obtain
$$\Phi_t^\gamma\nu=\L_{X_t^\nu}=\L_{X_t^\mu|\Q_t},$$
where $\L_{X_t^\mu|\Q_t}$ stands for the distribution of $X_t^\mu$ under $\Q_t$. Thus, by \eqref{xn2}, \eqref{ztn}, \eqref{N1} and H\"older's inequality, we find constant $C_4(N)>0$ such that
\begin{align*}
&\|\Phi^\gamma_t\mu-\Phi^\gamma_t\nu\|_{k,var}=\sup_{|f|\leq 1+\|\cdot\|_\tau^k}\left|\E\left[   f(X_t^\nu)-f(X_t^\mu)\right]\right| \\
&=\sup_{|f|\leq 1+\|\cdot\|_\tau^k}\left|\E\left[(R_t-1)  f(X_t^\mu)\right]\right|\leq \E\left[(1+\|X_t^\mu\|_\tau^k)|R_t-1|\right]\\
&\leq \E\left[\left\{\E\left((1+\|X_t^\mu\|_\tau^k)^2\big|\F_0\right)\right\}^{\ff 1 2}\left\{\E\left(|R_t-1|^2\big|\F_0\right)\right\}^{\ff 1 2}\right]\\
&\leq C_4(N)\e^{\t t}\tt\W_{k,\t,var}(\mu,\nu)\left( \int_0^tH(s)\e^{-2\t(t-s) }\,\d s\right)  ^{\ff 1 2},\ \ \mu,\nu\in\mathcal C^{\gamma,N}_k.
\end{align*}
Therefore, \begin{equation}\label{cons1}
\begin{aligned}
&\W_{k,\t,var}(\Phi^\gamma\mu,\Phi^\gamma\nu)=\sup_{t\in[0,T]}\e^{-\t t}\|\Phi^\gamma_t\mu-\Phi^\gamma_t\nu\|_{k,var}\\
&\leq C_4(N)\sup_{t\in[0,T]}\left( \int_0^tH(s)\e^{-2\t(t-s) }\,\d s\right)  ^{\ff 1 2}\tt\W_{k,\t,var}(\mu,\nu),\ \ \mu,\nu\in\mathcal C^{\gamma,N}_k.
\end{aligned}
\end{equation}
We will finish the proof by considering $k\le 1$ and $k>1$ respectively.

\paragraph{$(b_1)$} Let $k\le 1$. By 
\eqref{vark} and \eqref{cons1},   we obtain
\beg{align*} &\tt\W_{k,\t,var}(\Phi^\gamma\mu,\Phi^\gamma\nu)\le (1+c) \W_{k,\t,var}(\Phi^\gamma\mu,\Phi^\gamma\nu)\\
&\le (1+c)C_4(N) \left( \int_0^tH(s)\e^{-2\t(t-s) }\,\d s\right)  ^{\ff 1 2}\tt\W_{k,\t,var}(\mu,\nu),\ \ \mu,\nu\in\mathcal C^{\gamma,N}_k.\end{align*}
Noting that $H\in L_{loc}^1([0,\8);(0,\8))$ yields that 
\beq\label{LM} \lim_{\theta\to\infty} \sup_{t\in[0,T]}\left( \int_0^tH(s)\e^{-2\t(t-s) }\,\d s\right)  ^{\ff 1 2}=0,\end{equation}
we may choose $\theta>0$   large enough such that
$$\tt \W_{k,\t,var}(\Phi^\gamma\mu,\Phi^\gamma\nu)\le \ff 1 2 \tt\W_{k,\t,var}(\mu,\nu),\ \ \mu,\nu\in\mathcal C^{\gamma,N}_k.$$
This together with (a) implies that $\Phi^\gamma$ has a unique fixed point in $C^{\gamma,N}_k.$

\paragraph{$(b_2)$} Let $k>1$. Let $\Theta(t,\cdot):=id+u(t,\cdot)$ for  $u$ solving \eqref{PDE1}
such that \eqref{LAMU} and \eqref{UVV} holds.
Let
\beg{align*}
&Z(t):=   {X^\mu(t)+ \Theta(t,X^\mu(t))- X^\nu(t)- \Theta(t,X^\nu(t))},\\
&g_1(t):= \M\|\nn^2u\|(t,X^\mu(t))+ \M\|\nn^2u\|(t,X^\nu(t)),\\
& g_2(t):=g_1(t)+ \M\|\nn\si\|(t,X^\mu(t))+ \M\|\nn\si\|(t,X^\nu(t)),
\ \ \  t\in [0,T].
\end{align*}
By $(H_1)$-$(H_3)$, \eqref{UVV}, \eqref{Theta} and   It\^o's formula, we find a constant $c_1(k)>0$ such that
\begin{equation}\label{2jz}
\begin{aligned}
&\d \big\{\e^{2k\tau t}|Z(t)|^{2k}\big\}\leq  c_1(k)\left\{\e^{2k\tau t}\|Z_t\|_\tau^{2k}+H(t)\e^{2k\tau t}\W_{k,var}(\mu_t,\nu_t)^{2k}\right\}\d t\\
&+\e^{2k\tau t}|Z(t)|^{2k}\d A(t)+\d M(t), \ \ t\in[0,T],
\end{aligned}
\end{equation}
where \begin{equation}\label{mt1}\begin{split}
&A(t):= {c_1(k)}\int_{0}^{t}\left(1+g_2(s)^2 +g_1(s)\|X^\mu_s\|_\tau^{\alpha}\right)\,\d s,\\
&\d \<M\>(t)\leq  c_1(k)\left( 1+g_2(t)\right)^2\e^{4k\tau t}|Z(t)|^{4k}\,\d t.
\end{split}
\end{equation} Combining this with   \eqref{Theta}, \eqref{EPR},   the stochastic Gr\"onwall inequality in \cite[Lemma A.5]{B18} {, Lemma \ref{KEs}, Corollary \ref{cor1} and the fact that $X_0^\mu=X_0^\nu=X_0$,} we find a constant $c_2(k)>0$ such that
\begin{equation*}
\begin{split}
& \E\left[\sup_{s\in[0,t]}\e^{k\tau s}\|X^\mu_s-X^\nu_s\|_\tau^{k}\right]{= \E\left[\sup_{s\in[0,t]}\e^{k\tau s}|X^\mu(s)-X^\nu(s)|^{k}\right]}\\
&\leq\left(\E\left[\e^{\ff{3}{2}A(t)}\right]\right)^{\ff {1}3}\left(\E\left[\sup_{s\in[0,t]}\e^{-A(s)}\e^{2k\tau s}{|X^\mu(s)-X^\nu(s)|}^{2k}\right]^{\ff 3 4}\right)^{\ff 23}\\
&\leq c_2(k)\left(\int_0^tH(s)\e^{2k\tau s}\W_{k,var}(\mu_s,\nu_s)^{2k}\,\d s\right)^{\ff 1 2}.
\end{split}
\end{equation*}
Since  $\Phi_t^{\gamma}\mu=\L_{X^\mu_t}$ and $\Phi_t^{\gamma}\nu=\L_{X^\nu_t},$ by the inequalities above, we find a constant
$c_3(k)>0$ such that
\begin{align*}
&\W_{k,\t}(\Phi^\gamma\mu,\Phi^\gamma\nu)=\sup_{t\in[0,T]}\e^{-\t t}\W_{k}(\Phi_t^\gamma\mu,\Phi_t^\gamma\nu)\leq \sup_{t\in[0,T]}\e^{-\t t}\left( \E\left[\|X_t^\mu-X_t^\nu\|_\tau^k\right]\right) ^{\ff 1 k}\\
&\leq c_2(k)\sup_{t\in[0,T]}\left(\int_0^tH(s)\e^{-2k\t t}\W_{k,var}(\mu_s,\nu_s)^{2k}\,\d s\right)^{\ff 1 {2k}}\\
&		\leq c_3(k) \tt\W_{k,\t,var}(\mu,\nu)\sup_{t\in[0,T]}\left(\int_0^tH(s)\e^{-2k\t(t-s)}\,\d s\right)^{\ff 1 {2k}},\ \ \mu,\nu\in\mathcal C^{\gamma,N}_k,\ \theta>0.
\end{align*}
Combining this with \eqref{cons1} and \eqref{LM}, we may choose $\theta>0$ such that
\begin{align}
\tt\W_{k,\t,var}(\Phi^\gamma\mu,\Phi^\gamma\nu)\leq \ff 1 2 \tt\W_{k,\t,var}(\mu,\nu),\ \ \mu,\nu\in\mathcal C^{\gamma,N}_k.
\end{align}
Thus, $\Phi^\gamma$ has a unique fixed point in $\mathcal C^{\gamma,N}_{k}$.

\section{  Proof of Theorem \ref{T02}}
\begin{proof}[Proof of Theorem \ref{T02}(1)]
The well-posedness can be proved by using a standard fixed point theorem. Let
$X_0$ be $\F_0$-measurable with $\gg:=\L_{X_0}\in \scr P_2(\C_\tau)$. For any constant $\theta,T>0,$ the path space
$$\mathcal C_2^{\gg}:=\big\{\mu\in C^w([0,T]; \scr P_2(\C_\tau)):\ \mu_0=\gg\big\}$$
is complete under the metric
$$\W_{2,\theta}(\mu,\nu):=\sup_{t\in [0,T]}\e^{-\theta t}\W_2(\mu_t,\nu_t).$$ 
According to Theorem \ref{TB2}, $({H}')$ implies that for any $\mu\in \mathcal C_2^\gg$, the SDE 
\beq\label{XM} \d X^\mu(t)= b(t,X_t^\mu,\mu_t)\d t + \si(t,X^\mu_t)\d W(t),\ \ t\in [0,T], X_0^\mu=X_0\end{equation}
is well-posed with 
$$\Phi_\cdot^\gg\mu:= \L_{X_\cdot^\mu}\in \mathcal C_2^\gg.$$
For the well-posedness of \eqref{E-1}, it suffices to show that $ \Phi^\gg$ has a unique fixed point in $\mathcal C_2^\gg$.

Now, for $\nu_0\in \scr P_2(\C_\tau)$, for simplicity, we assume that we can choose $\F_0$-measurable random variables
$X_0^\mu$ and $X_0^\nu$ on $\C_\tau$ such that
\beq\label{M1} \W_2(\mu_0,\nu_0)^2= \E\big[\|X_0^\mu-X_0^\nu\|_\tau^2\big].\end{equation}
Otherwise, in the following it suffices to first replace $(X_0^\mu,X_0^\nu)$ be the sequences
$(X_0^{\mu,n}, X_0^{\nu,n})$ such that
$$n^{-1}+ \W_2(\mu_0,\nu_0)^2\ge \E\big[\|X_0^{\mu,n}-X_0^{\nu,n}\|_\tau^2\big],\ \ n\ge 1,$$
then let $n\to\infty$.

For any $\nu \in \mathcal C_2^{\nu_0}$ which is defined as $\mathcal C_2^\gg$ for $\nu_0$ replacing $\gg$, 
let $X_t^\nu$ be the unique solution to
\begin{equation}
\label{XM2} \d X^\nu(t)= b(t,X_t^\nu,\nu_t)\d t + \si(t,X^\nu_t)\d W(t),\ \ t\in [0,T]
\end{equation} with initial value $X_0^\nu$. 
By $(H')$ and It\^o's formula, we obtain
$$\d{\e^{2\tau t}}|X^\mu(t)-X^\nu(t)|^2\le K{\e^{2\tau t}}\big(\|X_t^\mu-X_t^\nu\|_\tau^2+\W_2(\mu_t,\nu_t)^2\big)\d t 
+\d M(t)$$
{for some constant $K>0$, }where $M_t$ is a martingale with 
$$\d \<M\>(t)\le K {\e^{4\tau t}}\|X^\mu_t-X_t^\nu\|_\tau^4 \d t.$$
Combining this with \eqref{EPR},  { It\^o's isometry and Young's inequality,} we find a constant $c_1>0$ such that
\beg{align*}&\E\bigg[\sup_{s\in [0,t]}{\e^{2\tau s}} \|X_s^\mu-X_s^\nu\|_\tau^2\bigg]
\le  c_1 \E\|X_0^\mu-X_0^\nu\|_\tau^2 + c_1\E\bigg[\bigg(\int_0^t {\e^{4\tau s}}\|X_s^\mu-X_s^\nu\|_\tau^4\,\d s\bigg)^{\ff 1 2}\bigg] \\
&\qquad \qquad\qquad \qquad\qquad \qquad +c_1 \int_0^t{\e^{2\tau s}}\E\big[\|X_s^\mu-X_s^\nu\|_\tau^2+\W_2(\mu_s,\nu_s)^2\big]\,\d s\\
& \le c_1\E\|X_0^\mu-X_0^\nu\|_\tau^2 + \Big(c_1+\ff {c_1^2}2 \Big) \int_0^t{\e^{2\tau s}}\E\big[\|X_s^\mu-X_s^\nu\|_\tau^2+ \W_2(\mu_s,\nu_s)^2\big]\,\d s\\
&\qquad \qquad\qquad \qquad\qquad \qquad+ \ff 1 2\E\bigg[\sup_{s\in [0,t]} {\e^{2\tau s}}\|X_s^\mu-X_s^\nu\|_\tau^2\bigg],\ \ t\in [0,T].\end{align*}
By an approximation argument with stopping times, we may and do assume that 
$$\E\bigg[\sup_{s\in [0,t]} {\e^{2\tau s}}\|X_s^\mu-X_s^\nu\|_\tau^2\bigg]<\infty,$$
so that the above estimate 
together with \eqref{M1} and Gr\"onwall's inequality yields
\beq\label{M2}\E\bigg[\sup_{s\in [0,t]} \|X_s^\mu-X_s^\nu\|_\tau^2\bigg]\le c_2(T) \W_2(\mu_0,\nu_0)^2+ c_2(T)\int_0^t \W_2(\mu_s,\nu_s)^2\,\d s\end{equation} for any $t\in[0,T]$ and some constant $c_2(T)>0$.
In particular, when $\mu_0=\nu_0=\gg$, we derive
\beg{align*}&\W_{2,\theta}(\Phi^\gg\mu,\Phi^\gg\nu)^2\le \sup_{t\in [0,T]}\e^{-\theta t} \E\big[  \|X_t^\mu-X_t^\nu\|_\tau^2\big]\\
&\le c_2(T) \sup_{t\in [0,T]}\e^{-\theta t} \int_0^t \W_2(\mu_s,\nu_s)^2\,\d s  
\le  \ff{c_2}\theta \W_{2,\theta}(\mu,\nu)^2,\ \ \mu,\nu\in \mathcal C_2^\gg.\end{align*}
So, when $\theta>c_2$, $\Phi^\gg$ is contractive in the complete metric space $(\mathcal C_2^\gg, \W_{2,\theta})$,
hence it has a unique fixed point.

Moreover, letting $\mu_t=P_t^*\mu, \nu_t=P_t^*\nu$, we have $\L_{X_t^\mu}=P_t^*\mu$ and $\L_{X_t^\nu}=P_t^*\nu$, so that    \eqref{M2} and Gr\"onwall's inequality yield
\eqref{*N} for some constant $c>0$.
\end{proof}
Now, let $\tau_0\in (0,\tau)$ and $\|\si\|_\infty+\|\si^{-1}\|_\infty<\infty,$
it remains to verify \eqref{log1}, which implies \eqref{lip} through repeating the proof of \cite[Theorem 2.1(1)]{BWY19} for 
\beq\label{GM2}  E=\scr P_{2,\e}^\alpha(\C_\tau),\ \rho(\mu,\nu)=\W_{2}(\mu,\nu),\ \Gamma_t=c\e^{-\tau_0t},\end{equation}
and $\Lambda_t=c\e^{ct}$ in place of $\Lambda$.
\beg{proof}[Proof of \eqref{log1}]
Let $X_t^\mu$ be the unique solution to \eqref{E-1} with the initial distribution $\mu$ and denote $\mu_t=P_t^*\mu$, $\nu_t=P_t^*\nu$, 
\begin{align*}
&\bar\zeta_s:=\left\{\si^*(\si\si^*)^{-1}\right\}(s,X^\mu_s)\left[b(s,X^\mu_s,\mu_s)-b(s,X^\mu_s,\nu_s)\right],\\
&\bar R_s:=\exp\left[-\int_0^s\<\bar\zeta_u,\d W(u)\>-\ff 1 2\int_0^s|\bar\zeta_u|^2\,\d u\right],\ \ s\in[0,t].
\end{align*}
Then $(H')$ and \eqref{*N} implies that \begin{align}\label{bzt}
|\bar\zeta_s|^2\leq c_1K_2\W_2(\mu_s,\nu_s)^2\leq c_2K_2\e^{c_2 s}\W_2(\mu,\nu)^2
\end{align}for some constants $c_1,c_2>0$. Thus by Girsanov's theorem, $\bar R_s$ is a martingale and
$$\bar W(s)=W(s)+\int_0^s\bar\zeta_r\,\d r,\ \ s\in[0,t]$$
is a Brownian motion under the probability measure $\bar\P:=\bar R_t\P$. Then \eqref{E-1} can be reformulated as
\begin{align}\label{eb}
\d X^\mu(s)=b(s,X^\mu_s,\nu_s)\d s+\si(s,X^\mu_s)\d\bar W(s), \ \ \L_{X^\mu_0}=\mu,\ s\in[0,t].
\end{align}
For any $\kappa>\tau$, where $\tau>0$ is given in \eqref{CR}, consider the following SDE:
\begin{equation}\label{Ys}
\begin{aligned}
\d Y(s)=&\{ b(s, Y_s,\nu_s)+\kappa\si(s, Y_s) \si(s,  X^\mu_s)^{-1}( X^\mu(s)-  Y(s))\}\d s\\
&+ \si(s, Y_s)\d\bar  W(s),\ \ s\in[0,t],\ \  Y_0=X_0^\nu.
\end{aligned} 
\end{equation}
Let  \begin{align*}
&\tt\zeta_s:=\kappa\si(s, X^\mu_s)^{-1}(X^\mu(s)- Y(s)),\\
&\tt R_s:=\exp\left[-\int_0^s\<\tt\zeta_r,\d\bar W(r)\>-\ff 1 2\int_0^s|\tt\zeta_r|^2\,\d r\right],\ \ s\in[0,t].
\end{align*} Due to \cite[Lemma 3.2]{BWY19} with $\P$ replaced by $\bar\P$, under the assumption $(H')$, $$\tt W(s)=\bar W(s)+\int_0^s\tt\zeta_r\,\d r=W(s)+\int_0^s\left(\tt\zeta_r+\bar\zeta_r\right)\,\d r,\ \ s\in[0,t]$$  is a Brownian motion under the probability measure
$\Q=\tt R_t\bar\P=\tt R_t\bar R_t\P$. Hence \eqref{Ys} can be reformulated as
\begin{equation}
\d Y(s)=b(s,Y_s,\nu_s)\d s+\si(s,Y_s)\d\tt W(s),\ \ s\in[0,t],\ \  Y_0=X_0^\nu,
\end{equation}
which together with the uniqueness of \eqref{XM2} derives that $\L_{Y_t|\Q}=\L_{X_t^\nu}$.
Moreover, if we choose $\F_0$-measurable random variables
$X_0^\mu$ and $X_0^\nu$ on $\C_\tau$ such that
\beq\label{M3} \W_2(\mu,\nu)^2= \E\big[\|X_0^\mu-X_0^\nu\|_\tau^2\big].\end{equation}
By \eqref{eb}, $(H')$ and   the proof of \cite[Lemma 3.3]{BWY19},    for any $p>0$ and $\tau_0\in(0,\tau)$ we find  a constant  $\kk>\tau$ to define $Y_s$ in \eqref{Ys}  such that
\begin{equation}\label{ttq}
\E_{\Q}\left[\|X^\mu_t-Y_t\|_{\tau}^p|\F_0\right]\leq c\e^{-p\tau_0t}\|X^\mu_0-X^\nu_0\|_{\tau}^p,\ \ t\geq 0
\end{equation}  holds for some constant $c>0$.
Therefore, applying Young's inequality in \cite[Lemma 2.4]{ATW}, 
\begin{equation}\label{log3}
\begin{split}
&P_t\log f(\nu)=\E_{\Q}\left[\log f(Y_t)\right]=\E_{\Q}\left[\log f(X^\mu_t)\right]+\E_{\Q}\left[\log f(Y_t)-\log f(X^\mu_t)\right]\\
&\leq \E\left[\bar R_t\tt R_t \log f(X^\mu_t)\right]+\|\nn\log f\|_\8\E_{\Q}\|X_t^\mu-Y_t\|_\tau\\
&\leq \E\left[\bar R_t\tt R_t\log\left(\bar R_t\tt R_t\right)\right]+\log P_tf(\mu)+c\|\nn\log f\|_\8\e^{-\tau_0t}\W_{2}(\mu,\nu), \ \ t\geq 0
\end{split}
\end{equation}
holds for any $f\in\B_b^+(\C_r)$ with $\|\nn \log f\|_\8<\8$. Next, denote $R_t=\bar R_t\tt R_t$, it follows from \eqref{bzt}, \eqref{ttq} that for some positive constants $c_3,c_4$,
\begin{equation*}
\begin{aligned}
&\E\left[R_t\log R_t\right]\leq \ff 1 2\E_{\Q}\int_0^t|\bar\zeta_s+\tt\zeta_s|^2\,\d s
\leq \E_\Q\int_0^t|\bar\zeta_s|^2\,\d s+\E_\Q\int_0^t|\tt\zeta_s|^2\,\d s\\
&\leq c_3K_2\e^{c_3 t}\W_2(\mu, \nu)^2+c_3\int_0^t\E_{\Q}\|X^\mu_s-Y_s\|_{\tau}^2\,\d s\leq \left(c_3K_2\e^{c_3 t}+c_4\right)\W_2(\mu, \nu)^2.
\end{aligned}
\end{equation*}
Substituting this back into \eqref{log3} yields \eqref{log1}.
\end{proof}
\section*{Acknowledgements}
The authors would like to thank the associated editor and referees for their helpful comments and suggestions.

\end{document}